\documentclass[psamsfonts,12pt]{amsart}
\usepackage{amsmath}
\usepackage{amssymb}
\usepackage{enumerate}
\usepackage{graphicx, epsfig, psfrag}

\let\cal\mathcal
\def \mR{\mathbb R}
\def \Z{\mathbb Z}

\def \T{\mathbb T}
\def \C {\mathbb C}
\def \Q {\mathbb Q}

\def \Zk{\mathbb Z^k}
\def \Zm{\mathbb Z^m}
\def\Zmm{\mathbb Z^{m-1}}

\def \Z2{\mathbb Z^2}
\def \Z3{\mathbb Z^3}

\def \Rk {\mathbb R^k}
\def \Rm {\mathbb R^m}

\def \R2{\mathbb R^2}
\def \R3{\mathbb R^3}
\def \Rsc{\mathbb R^{s(\cC)}}				
\def \Rs{\mathbb R^s}					
\def \Ru{\mathbb R^u}					

\def\Tmp{\mathbb T_\pm^m}
\def \Tk{\mathbb T^k}
\def \Tm{\mathbb T^m}
\def \Tk{\mathbb T^k}
\def \T2{\mathbb T^2}
\def \T3{\mathbb T^3}

\def \C2{\mathbb C^2}
\def \C3{\mathbb C^3}

\def \S1{\mathbb S^1}
\def \S2{\mathbb S^2}
\def \S3{\mathbb S^3}

\def\a{\alpha}
\def\ao{\alpha_0}
\def\b{\beta}

\def\l{\lambda}

\def\n{\nu}

\def\S{\Sigma}

\def \an{{\alpha (\mathbf n)}}

\def \w{\cal W}
\def \cC{{\mathcal C}}

\def\0C{\mathcal C^0}
\def\1C{\mathcal C^1}
\def\2C{\mathcal C^2}

\newcommand{\W}{\mathcal{W}}

\def\E2P {\exp 2\pi i}

\newcommand{\Ker}{\operatorname{Ker}}

\newcommand{\Id}{\operatorname{Id}}

\newcommand{\bm}{{\mathbf m}}

\newcommand{\bn}{{\mathbf n}}

\newcommand{\diag}{\mathop{\rm diag}}

\newcounter{Theorem}

\newtheorem{thm}{Theorem}
\newtheorem{coro}{Corollary}
\newcounter{Lemma}

\newtheorem{lem}{Lemma}[section]
\newtheorem{cor}[lem]{Corollary}
\newtheorem{prop}[lem]{Proposition}

\theoremstyle{plain}
\newtheorem*{THM}{Main Theorem}

\theoremstyle{definition}

\newtheorem{rem}{Remark}

\numberwithin{equation}{section}

\newcommand{\foot}[1]{\mbox{}\marginpar{\raggedleft\hspace{0pt}\tiny #1}}

\begin{document}
\title[]{Arithmeticity  and topology of smooth  actions of   higher rank abelian groups }
{}
\author[]{ Anatole Katok $^1)$ and Federico Rodriguez Hertz $^2)$}


\subjclass[2003]{}

\address{Department of Mathematics\\
        The Pennsylvania State University\\
        University Park, PA 16802 \\
        USA}
\email{katok$\_$a@math.psu.edu}
\email{hertz@math.psu.edu}
\urladdr{http://www.math.psu.edu/katok\_a }
 
\thanks{ $^1)$ Based on research  supported by NSF grant DMS 1002554.}

\thanks{ $^2)$ Based on research  supported by NSF grant DMS 1201326}

\begin{abstract} We prove that any smooth  action of $\Zmm, \,m\ge 3$ on an $m$-dimensional manifold that preserves a measure such that all non-identity elements of the suspension have positive entropy  is essentially algebraic, i.e. isomorphic up to  a finite permutation to an affine action on the torus or its factor by $\pm\Id$. Furthermore this isomorphism  has nice geometric properties, in particular, it is smooth  in the sense of Whitney on a set whose complement has arbitrary small measure.  We further derive restrictions on topology of manifolds that may admit such actions, for example, excluding spheres and obtaining below estimate on the first Betti number in the odd-dimensional case. 
\end{abstract}
\maketitle

\section*{Introduction}
Let $\alpha$ be a smooth action of of $\Zmm,\, m\ge 3$ on a compact  $m$-dimensional manifold  $M$, not necessarily compact.  We assume that $\alpha$ is uniformly $C^{1+\theta},\, \theta>0$ with respect to a certain  smooth  Riemanninan metric on $M$,  i.e. the generators of the action and their inverses have uniformly bounded derivatives satisfying  Hoelder condition with exponent $\theta$ and a fixed Hoelder constant.  Naturally, if $M$ is compact this condition does not depend on the choice of  the Riemannian metric. This regularity assumption allows to apply standard results of smooth ergodic theory to any invariant measure of the action. 

Following  \cite{KKRH}  we assume that  $\alpha$ has an invariant probability ergodic measure $\mu$   such that 
\begin{enumerate} 
\item\label{(1)} Lyapunov characteristic exponents are non-zero and are  in  general position, i.e.   the dimension of the intersection of any $l$ of  their kernels  is the minimal possible, i.e. is equal, 
to $\max\{ k-l, 0\}$, 
\item\label{(2)} at least one element
in $\Zmm$ has positive entropy with respect to $\mu$.
\end{enumerate} 

We will call such a pair $(\a, \mu)$ a {\em maximal rank positive entropy action}. 

The main result of \cite{KKRH} is absolute continuity of the  measure $\mu$ for a maximal rank positive entropy action. 
In \cite[Sections 8.1, 8.2]{KKRH}  a program  of further study  of such actions has been formulated. 

In the present paper  we mostly complete this program. 
Firstly we extend   the  description of  maximal rank actions on the torus with Cartan homotopy data  from \cite{KRH1},
where  a positive  entropy  hyperbolic measure always exists, to maximal rank positive entropy actions on arbitrary manifolds.  Secondly we obtain  substantial information on  topology of manifolds that  may admit such actions., in particularly excluding spheres and many other  standard manifolds. 

Let us call  an {\em infratorus} a factor of $\Rm$  by a group $E$  of affine transformations 
that contains a lattice $L$  of translations  as a finite index subgroup. Thus an infratorus is  the factor of the torus  $\Rm/L$ by a finite group $G$  of  affine transformations. In this definition  infratorus is a varifold and not necessarily a smooth manifold since the group $G$  may not act freely; in particular it may have fixed points. In fact, the only examples of   infratori that  admit maximal  rank abelian actions by affine transformations  and that are not tori,   are of that  kind:  such an infratorus is obtained by factorizing  $\Tm$  by the involution $Ix=-x$ that has $2^m$ fixed points. Let us denote  such an infratorus $\Tmp$.

\begin{rem}By blowing up the singular points and glueing 
in copies of the projective space of codimension one, one  constructs a smooth action on a manifold  that is  diffeomorphic to the affine action on the infratorus outside of the singular points, see \cite{KL}. This can be considered as the ``standard smooth model'' of the infratorus action.    Examples of infratori that are smooth manifolds can be found in \cite[Section 2.1.4]{KNbook}. \end{rem}

We  formulate and present  our results in two parts.  The reason is that the first part is likely to hold with proper modifications (in particular, allowing more general infratori) in greater generality,  namely, in the setting similar to that of \cite{KRH2} where no connection is assumed between the rank $k\ge 2$ of the action    and the dimension of the ambient manifold.
Most steps in the proof work in that generality  and remaining difficulties, while substantial, are of technical nature.  
The second part heavily relies on existence  of codimension-one stable manifolds and hence  is specific for the  maximal rank setting.

The   first part    (Theorem~\ref{TMain1})  states in particular that  modulo a finite permutation any such action is  ``arithmetic'', 
i.e. there is a measurable isomorphism  between  the restriction of the action to each ergodic component of a certain finite index subgroup $\Gamma\subset\Zmm$  and a Cartan action  by affine automorphisms of the torus $\Tm$ or its factor $\Tmp$.  This isomorphism has nice topological and  geometric properties that are described in detail below. 
This provides solutions of  most conjectures and problems from  \cite[Sections 8.1]{KKRH}. 

The second part  asserts  that   the restriction of the above mentioned isomorphism to each ergodic component of the group $\Gamma$   extends   to a continuous map between an open set in $M$  and the  complement  to a finite set on $\Tm$ or $\Tmp$ that is a topological semi-conjugacy (a factor-map) between $\a$ and $\ao$  (Theorem~\ref{TMain2}). Furthermore, this map can be 
modified on a set of arbitrary small measure and then  extended to a  homeomorphism between an open set in $M$  and the  complement  to a finite set on $\Tm$ or $\Tmp$.  
This has implications for the topology of $M$, in particular disproving Conjecture 4 from \cite{KKRH}.



Technically the present  paper  builds upon the results of \cite{KKRH, KRH2}. We use  background information from those papers without special references.

\section{Formulation of results}
 
\subsection{The arithmeticity  theorem}

\begin{thm}\label{TMain1}  Let $\a$ be a $C^r, \,1+\theta\le r\le  \infty $ maximal rank positive entropy action 
 on a smooth manifold $M$  of dimension $m\ge 3$.

Then there  exist:
\begin{itemize}

\item  disjoint measurable sets of equal measure  $R_1,\dots, R_n\subset M$ such that $R= \bigcup_{i=1}^n R_i $ has full measure and the action $\a$  cyclically interchanges those sets. Let $\Gamma\subset \Zmm$ be the stationary subgroup of any of the sets $R_i$ ($\Gamma$ is of course isomorphic to $\Zmm$);
 
\item  a Cartan action  $\ao$ of $\Gamma$   by affine transformations   of either the torus $\Tm$ or the infratorus  $\Tmp$ that we will call  {\em the algebraic model}; 
  
 \item measurable maps  $h_i: R_i\to\Tm$ or $h_i: R_i\to\Tmp, \, i=1,\dots, n$; 
 \end{itemize}
 such that
  
\begin{enumerate}

\item \label{TM1} $h_i$ is bijective almost everywhere and $(h_i)_*\mu=\l$, the Lebesgue (Haar) measure on $\Tm$ (correspondingly $\Tmp$);

\item \label{TM2} $\ao\circ h_i=h_i\circ \a$;

\item \label{TM3} for almost every $x\in M$  and every $\bf n\in \Zmm$ the restriction of $h_i$ to the 
stable manifold $W^s_x$ of $x$ with respect to $\a({\bf n})$ is  a $C^{r-\epsilon} $ diffeomorphism for any $\epsilon>0$.

\item \label{TM4} $h_i$ is  $C^{r-\epsilon}$ in the sense of Whitney on a set  whose complement  to $R_i$ has arbitrary small measure; those sets will be described in the course of proof; in particular, they are saturated by  local stable manifolds.

\end{enumerate}
\end{thm}

Sometimes, when this cannot cause confusion,  we will call actions satisfying assumptions of Theorem~\ref{TMain1} simply {\em maximal rank actions}.


\begin{rem} Statement \eqref{TM4} implies that the measure $\mu$ is absolutely continuous. However, as we  mentioned before, this fact  is  the principal result of \cite{KKRH}  and it forms a basis of the proof of Theorem~\ref{TMain1}. 
\end{rem}

\begin{rem} Statements \eqref{TM1} and \eqref{TM2} imply that $h_i$ is a measurable isomorphism between 
$(\a,\mu)$,  restricted to  the set  $R_i$ and subgroup $\Gamma$,    and the algebraic model $(\ao,\l)$.
\end{rem} 

\begin{rem} It follows form  \eqref{TM1} and \eqref{TM2} that  for $n=1$ the action $\a$ is  weakly mixing (and, in fact, mixing); for $n> 1$  or equivalently, if $\Gamma\neq \Zmm$,  $\Gamma$ action is not ergodic 
but its restriction  to any of its $n$ ergodic components is weakly mixing (and mixing). 
\end{rem}

\begin{rem} Statement  \eqref{TM3}   immediately  implies that  Jacobians  along Lyapunov  foliations 
are   measurably cohomologous to constants (exponents for the algebraic model), thus solving Conjecture 1 from \cite{KKRH}.  Furthermore, the transfer function is smooth along the Lyapunov foliations. 
\end{rem} 


\subsection{Corollaries from arithmeticity}

\subsubsection{Entropy and Lyapunov exponents}
Theorem~\ref{TMain1}   immediately implies  the solution of Problem 1 and Conjecture 2 from \cite{KKRH}. In fact,  
 description of  Cartan (maximal rank) actions on the torus via units in  the algebraic number fields  given in \cite{AKSKKS} provides more precise information. We consider the weakly mixing case first. 
\begin{coro}\label{coro1} 
Let $\a$ be a $C^{1+\theta}, \theta>0$ weakly mixing  maximal rank positive entropy $\Zmm$ action. 
There exists a totally real  algebraic number field  $K$ of degree $m$, 
that is a simple extension of $\Q$  uniquely determined by $\a$, and,  for any system of generators   of $\a$, an  $(m-1)$-tuple of multiplicatively independent units  $\lambda_1,\dots, \lambda_{m-1}$ in $K$ such that the Lyapunov characteristics  exponents for  those  generators of $\a$  are $$\log|\phi_1(\lambda_i)|,\dots ,|\log\phi_{m}(\lambda_i)|, \,\, i=1,\dots, m-1. $$
where $\phi_1, \dots, \phi_{m}$ are different embeddings of $K$ into $\mathbb R$. 
\end{coro}

In the  general case one applies Corollary~\ref{coro1} to the restriction of the action to the stationary 
subgroup $\gamma$ for each of the sets $R_i$.  Those restrictions  for $i=1,\dots, n$ are isomorphic and hence 
have the same entropy that is also equal  to the entropy of $\a(\gamma)$ for any $\gamma\in\Gamma$ with respect to the non-ergodic measure $\mu$. Let $k$  be the index of $\Gamma$. Since $k$-th power of any element of $\Zmm$ lies in $\Gamma$ and every element is a power of a generator  one immediately obtains description of exponents in the general case. 

\begin{coro} Lyapunov exponents of any element of a maximal rank action $\a$  have the form 
$$ \frac{|\log\phi_i(\lambda)|}{k},\dots, \frac{|\log\phi_{m}(\lambda)|}{k}$$
where $\lambda$ is a  unit in a totally real  algebraic number field   of degree $m$, and $k$ is a positive integer that depends only on $\a$ but not on the element. Here as before $\phi_1, \dots, \phi_{m}$ are different embeddings of the field into $\mathbb R$. 
\end{coro}

Since entropy of an action element is equal to the Mahler measure of the corresponding unit we can use exponential below estimate for the Mahler measure for totally real fields \cite{Schnizel, HS} to obtain a lower bound on entropy.

\begin{coro}\label{entropyboundWM} The entropy of any element of  a weakly mixing maximal entropy action on an $m$-dimensional manifold is bounded from below by $cm$, where $c$ is a universal constant.
\end{coro}

\subsubsection{Entropy and isomorphism rigidity} 
 Eigenvalues of an integer matrix $A$,  when simple and real,  determine  
its conjugacy class over $\mR$ and hence over $\Q$.  Assume that $\det A=\pm 1$. That in turn determines   a conjugacy  class of the automorphism of the torus $F_A$  up to a common finite factor or finite extension.  By \cite[Theorem 5.2]{AKSKKS} for a broad class of   $\Zk, k\ge 2$ actions  by
automorphisms of a torus, that includes all  Cartan   actions,  measure theoretic isomorphism (with respect to Lebesgue measure) implies algebraic   isomorphism. 

Notice that passing to a finite factor or finite extension  does not change entropy. Likewise the entropy for affine actions with the same linear parts are the same.  By symmetry  the entropy does not change if all generators of an action are replaced by their inverses. Theorem~\ref{TMain1}  allows to show that  entropy function determines 
a maximal rank action action on a finite index subgroup  up to a measurable isomorphism  with above mentioned 
trivial modifications. 
We call  the next statement  a corollary,  despite  the length of the   argument needed to deduce it from Theorem~\ref{TMain1}. The point is that the argument is  purely algebraic and  deals  only with linear actions, modulo choosing  appropriate finite index subgroups.

\begin{coro}\label{entropyrigidity} Let $\a$ and $\a'$ be two maximal rank actions. Assume that they are both weakly mixing and their entropy functions coincide.  Then restrictions of   $\a$ and $\a'$ to a certain subgroup $\Gamma\subset \Zmm$ of finite index are finite factors  of measurably isomorphic actions, possibly with replacing all generators of one action by their inverses. 
\end{coro}

\begin{proof} Let $\ao$ and $\ao'$  be the  algebraic  models for $\a$ and $\a'$. Weak mixing 
implies that for both of them $n=1$. Now take finite covers  $\tilde{\ao}$ and $\tilde{\ao'}$ (if necessary) 
 that are actions by affine  transformations of  $\Tm$.  Take finite index subgroups  of $\Zmm$ for which $\tilde{\ao}$ and $\tilde{\ao'}$ act by automorphisms.  Taking the intersection  of those subgroups   obtain a finite index subgroup $\Gamma_1$ for which both $\tilde{\ao}$ and $\tilde{\ao'}$ act by automorphisms. Those restrictions still have identical entropy functions.  Now there is a subgroup $\Gamma_2$ of finite index such that eigenvalues for all generators of both actions are positive. Let $\Gamma=\Gamma_1\cap\Gamma_2$. Since eigenvalues are simple they are thus determined for the $\Gamma$  action by Lyapunov exponents. But by \cite[Proposition 3.8]{AKSKKS}
irreducible (in particular, Cartan)  actions by automorphisms with the same eigenvalues of their generators are
algebraically conjugate  to finite factors of the same action. 
 
Let us show that  Lyapunov exponents are in turn determined by the entropy function, possibly with replacing all generators by their inverses. To see that, notice first that entropy function is not differentiable exactly at the union of kernel of the Lyapunov exponents, the Lyapunov hyperplanes. Thus it determines every Lyapunov exponent up to a scalar multiple. In the Cartan case for each Lyapunov exponent $\chi$  there is exactly one Weyl chamber $\mathcal C_\chi$ where this exponent is positive and  all other negative. Inside this Weyl chamber  entropy is equal to $\chi$. This Weyl chamber  and its opposite $-\mathcal C_\chi$ are determined from the configuration of Lyapunov hyperplanes as the only two whose boundaries intersect all Lyapunov hyperplanes except for $\Ker\chi$. Thus for every Lyapunov exponents  $\chi$ of $\alpha$ there is en exponent $\chi'$ of $\a'$
such that  is equal to either $\a$ or  $-\a$. Let us show that for all exponents the sign is the same. For, suppose
that for two exponents $\chi_1$ and $\chi_2$ of $\a$ there are exponents $\chi_1$ and $-\chi_2$ of $\a'$. 
then in the Weyl chamber $\mathcal C_{\chi_{1}}$ $\chi_2$ is negative, hence in this Weyl chamber the entropy
of $\a'$ is at least $\chi_1-\chi_2$, i.e greater than the entropy of $\a$. 
Thus all exponents of $\a$ and $\a'$ are either equal or have opposite signs. In the latter case we  can change generators of $\a'$ to their inverses and obtain actions with  equal exponents.  \end{proof}

In the case of maximal rank actions that are not weakly mixing one restricts the action to the stationary subgroup $\Gamma$  for each of the sets $R_i$  and to apply Corollary~\ref{entropyrigidity}  to each of those  sets. 
An obvious additional invariant is the index of $\Gamma$. It can be determined, for example, from 
 the discrete spectrum of the action. This spectrum  determines 
how different elements of the action interchange the sets $R_i$ A conjugacy  between restrictions of the action of $\gamma$ to different sets $R_i$ can be effected by  using the action of elements from corresponding  cosets of $\Gamma$. Hence  Corollary~\ref{entropyrigidity} can be simultaneously applied to those sets. 
\begin{coro}\label{entropyrigidityNWM} Let $\a$ and $\a'$ be two maximal rank actions. Assume that they have the same discrete spectrum and their entropy functions coincide.  Then restrictions of   $\a$ and $\a'$ to a certain subgroup $\Gamma\subset \Zmm$ of finite index are finite factors  of measurably isomorphic actions, possibly with replacing all generators of one action by their inverses. 
\end{coro}

\subsubsection{Cocycle rigidity} Proper classes of cocycles  over actions that we consider are Lyapunov Holder and Lyapunov smooth, i.e. measurable cocycles that are  Holder or smooth correspondingly along Lyapunov foliations almost everywhere;    see \cite[Definition 8.1]{KRH2} for precise definition. Any such cocycle over a $C^{\infty}$ maximal rank action can be transferred to a cocycle over a finite extension of the linear model the same way as is described in the proof of \cite[Theorem 2.8]{KRH2}. This  proof works verbatim in our case and  produces cocycle rigidity.

\begin{coro} Any Lyapunov Holder (corr. Lyapunov smooth)  real valued  cocycle over a $C^r, 1+\theta\le r \le \infty$ maximal rank action is cohomologous to a constant cocycle  via a Lyapunov Holder (corr. Lyapunov smooth) transfer function  (with the obvious proviso that  the Lyapunov regularity of the transfer function  is less that regularity  of the action). 
\end{coro}

\subsection{The topology theorem}
Let $L$ denotes either a torus $\Tm$ or the infratorus $\Tmp$; 


 \begin{thm}\label{TMain2} Let $\a$ be a $C^r, \,1+\theta\le r\le  \infty $ maximal rank positive entropy action, then 
 \begin{enumerate}
 \item the sets $R_1,\dots, R_n$ in Theorem \ref{TMain1} can be chosen inside  open $\Gamma$-invariant subsets
 $O_1,\dots, O_n$, that are also  interchanged by $\a$;
 \item  each map  $h_i$ extends to a continuous map  $\tilde h_i: O_i\to L\setminus F$ where 
 $F$ is a finite $\ao$-invariant set; 
 
 \item if $L$ is a torus or if $x\in L$  is  a regular  point in the infratorus  then  there exists an arbitrary small parallelepiped $P_x$ (in some linear coordinates) containing $x$ such that on the boundary of $P_x$ the map $h_i$ is invertible  and  the inverse is a diffeomorphism  on every  face of $P_x$;

\item\label{topinfra}  if  $x\in F$  is  a singular  point in the infratorus  $\Tmp$ then  there exists an arbitrary small projective parallelepiped $P_x$ (the factor of a centrally symmetric parallelepiped in  some linear coordinates  by the involution $t\to -t$) containing $x$ such that on the boundary of $P_x$ the map $h_i$ is invertible  and  the inverse is a diffeomorphism  on every  face of $P_x$;

\item let $\mathcal R=L\setminus \bigcup_{x\in F}{\rm Int} P_x$. Then $h_i^{-1}\mathcal R$ is homeomorphic to $\mathcal R$ via a homeomorphism $H$ that  coincides with $h_i$ on $\partial P_x$.
 \end{enumerate}
 
 \end{thm}
\begin{rem} Notice that any singular point in $L$ must be  in the exceptional set $F$ because topology of a small neighborhood  of a singular point is different from that   of  points in a  manifold. 
\end{rem}

\subsection{Topological corollaries} 
 Theorem~\ref{TMain2}  allows to make conclusions about topology of   manifolds that admit  maximal rank positive entropy  actions. Right now we list only some of those properties that can be derived
quickly. More detailed discussion of the consequences of Theorem~\ref{TMain2} will appear in a separate paper.  

\begin{coro}\label{corTMain2} Let $M$ be a connected  manifold of odd dimension  $m\ge 3$  that admits a maximal rank positive entropy action of  $\mathbb Z^{m-1}$.  
Then if   $M$  is orientable  it is homeomorphic to  the connected sum 
of the torus $\Tm$   with another  manifold. If $M$ is non-orientable, its orientable double cover  is homeomorphic to  the connected sum 
of the torus $\Tm$   with another  manifold.

In particular,  in both cases the fundamental group $\pi_1(M)$ contains a subgroup isomorphic to $\Zm$.
\end{coro}
\begin{proof} Consider the orientable case first. 
In odd dimension the infratorus $\Tmp$ is not orientable and the same is true to its complement to a finite set or, equivalently, to  the complement to the union of finitely many small balls. Since an open subset of an orientable manifold is orientable the open subset $S= {\rm Int}\, h_i^{-1}\mathcal R\subset M$ is  orientable and hence $L$ is the torus. Now take a disc $D\subset\Tm$ that contains the set $F$ and consider closed set $H^{-1}(\Tm\setminus D)$. Its boundary is a sphere and thus $M$ is the connected sum of $\Tm$ and a manifold that is obtained by glueing a disc to the boundary of $M\setminus H^{-1}(\Tm\setminus D)$.

Now assume that $M$ is non-orientable and take the orientable  double cover $\tilde M$ of $M$. 
Let $I: \tilde M\to\tilde M$ be the deck transformation.  Consider lifts of the elements of the  maximal rank action $\a$ to $\tilde M$. Each element  $f$ has two lifts $f_1$ and $f_2=f_1I$. The involution $I$ commutes with all lifts. 
The group $\Gamma$ consisting of all lifts is either  abelian or  its commutator is the group of two elements generated by $I$

Let us  show that $\Gamma$ has a finite index abelian subgroup  isomorphic to $\Zmm$. If $\Gamma$ is already abelian then it is isomorphic to the direct product $\Zmm\times\mathbb Z/2\mathbb Z$. Otherwise consider generators of the action $\a$ and let $f_1,\dots,f_{m-1}$ be their lifts to $\tilde M$. Centralizer $Z(f_i)$ of each of those elements in $\Gamma$ is either the whole of $\Gamma$, or an index two subgroup. This follows from the fact $I^2=\Id$ and that $I$ is in the center of $\Gamma$ since this implies that the product of any two elements not in $Z(f_i)$  belongs to $Z(f_i)$. Thus $\mathcal Z=\bigcap_{i=1}^{m-1}Z(f_i)$ is a finite index abelian subgroup of $\Gamma$ that belongs to its center. Notice that the index of $\mathcal Z$ in $\Gamma$ is at most $2^{m-1}$.  Since the only finite order element of $\Gamma$ is $I$, $\mathcal Z$ is isomorphic to $\Zmm\times\mathbb Z/2\mathbb Z$. 

Thus $\Zmm$ acts on $\tilde M$ by lifts of elements of $\a$. This is obviously a maximal rank positive entropy action so that from the argument for the orientable case $\tilde M$ is the connected sum of torus with another manifold.  Since $\pi_1(\tilde M)$ embeds into $\pi_1(M)$ the  former a subgroup isomorphic to $\Zm$.
\end{proof}

A very similar argument allows to  partially extend Corollary~\ref{entropyboundWM} to actions that are not weakly mixing. 

\begin{coro}The entropy of any element of a weakly mixing maximal 
entropy action on an $m$-dimensional manifold  $M$ is bounded from below 
by $\frac{cm^2}{\beta_1(M)}$, where c is a universal constant and $\beta_1$ is the first Betti number.
\end{coro}
\begin{proof} As before, let us consider the orientable case first. The obvious  below estimate the entropy  for the  elements  of   the action of $\Gamma$ on each ergodic component, divided by the  number $n$ of ergodic components that is equal to the index of $\Gamma$.  Hence this number needs to be estimated from above.  
repeating the argument about connected sums from the proof of Corollary~\ref{corTMain2} we deduce that $M$ is homeomorphic of the connected sum of $n$ copies of the torus $\Tm$ and another manifold. Hence by Mayer-Vietoris theorem $\b_1(M)\ge mn$. Now Corollary~\ref{entropyboundWM} implies the  needed estimate.

In the non-orientable case we consider the orientable double cover, lift the action as in the proof of Corollary~\ref{corTMain2} notice that entropy does not change and use the estimate in the orientable case.  Since the Betti number of the manifold is the same as of the orientable double cover, the inequality follows. \end{proof}

\begin{rem} Notice that there  is no estimate from below that depends on dimension only as in the weak mixing case. An appropriately modified version of the suspension construction over a weakly mixing action on the torus involving maxing holes around fixed points similarly to \cite{KL}  and  connecting them by cylinders, similarly to the filling of holes descried in \cite{KRH1}, produces examples with arbitrary low entropy.  
\end{rem}

Even-dimensional case is more complicated. While the case $L=\Tm$ of course  works the same way, if $L=\Tmp$
the manifold $M$ may not be  a connected sum with the infratorus as one of the components. Since in this case the infratorus  is orientable the double cover trick does not work. Indeed, there are some simply-connected manifolds, for example, some $K3$ surfaces, that admit maximal rank actions.  Still some conclusions can be drawn. Here is a simple example. 

\begin{coro}Maximal rank positive entropy actions to not exist on any sphere.
\end{coro}

\begin{proof} Only the case $L=\Tmp$ needs to be considered. In this case by Theorem~\ref{TMain2}\eqref{topinfra}
there exists a smooth embedding of $\mathbb RP(m-1)$ to $M$. If $M=S^m$  this would imply existence 
on am embedding into $\Rm$ that is impossible \cite{Hopf}.
\end{proof}

\subsection{Structure and general remarks on  the proof} 

 We  begin with  \cite[Theorem 2.11]{KRH2}, (that is based   on the main technical Theorem 4.1 from \cite{KKRH}),  that  states that in our setting conditional measures of the leaves on  each Lyapunov foliation are absolutely continuous. 
This implies   that the measure  $\mu$ itself is absolutely continuous, see \cite[Theorems 5.2 and 2.4]{KRH2}. 

Furthermore, there is a unique measurable system of smooth affine parameters  on the leaves of a Lyapunov foliation \cite[Proposition 3.3]{KRH2}. In fact, these affine parameters and conditional measures are closely related: the  affine parameter is obtained by integrating the conditional measure. 

The next fact  used in the proof is \cite[Proposition 4.2]{KRH2} that asserts that conditional measures (and hence affine structures) are invariant  (in the affine sense) with respect to holonomy  along  complimentary directions that includes all remaining Lyapunov foliations. 

\begin{rem} While absolute continuity of conditional measures of an absolutely continuous hyperbolic measure,  as well as connection between conditional  measures and affine structures,  are fairly general facts, holonomy invariance  is  a specific higher rank phenomenon: it fails already for  non-algebraic area-preserving Anosov diffeomorphisms of  two-dimensional torus.
\end{rem} 

Naturally, both conditional measures and affine parameters are  defined up to a scalar multiple. Fixing a measurable normalization smooth along the leaves of the Lyapunov foliation in question  produces a cocycle; different normalizations produce cohomologous cocycles. 


\begin{rem}Notice that in the proof of \cite[Theorem 4.1]{KKRH} we use a special time change for the suspension action and for this time change  the expansion coefficient of  the original suspension action  in the Lyapunov direction is indeed cohomologous to a constant.   This however does not imply that  expansion coefficient for the original  action or its suspension is cohomologous to a constant.  For example this is not the case  for hyperbolic flows  where W. Parry constructed  synchronization time change that inspired our construction. \cite{Parry}
\end{rem}

\section{Affine structures and Holonomies} 

\subsection{Weak-mixing reduction}\label{weakmixingreduction}

Let us start with the weakly mixing  reduction which gives the decomposition claimed in Theorem \ref{TMain1}. Let $\a:\Zmm\to Diff(M^{m})$ be a $C^{1+\theta}$ action as in Theorem \ref{TMain1}. Take $\bn\in\Zmm$ such that $\mu$ is a hyperbolic measure for $\an$ absolutely continuous w.r.t. Lebesgue measure. By Pesin Ergodic decomposition Theorem \cite{P} there is $k>0$ and a $\a(k\bn)$-invariant set $R_1\subset M$ of positive $\mu$-measure such that $\a(k\bn)|R_1$ is a Bernoulli automorphism, in particular weakly mixing. Set $\bn_1:=k\bn$. By ergodicity of $\a(\bn_1)|R_1$ we have that for any $\bm\in\Zmm$, either $\mu(\a(\bm)(R_1)\cap R_1)=0$ or $\a(\bm)(R_1)\stackrel{o}{=}R_1$.

 Since $\mu$ is an ergodic invariant measure for the whole action there are $\bn_2,\dots\bn_n\in\Zmm$ such that $\a(\bn_l)(R_1)\cap\a(\bn_k)(R_1)\stackrel{o}{=}\emptyset$ for $k,l=1,\dots n$, $k\neq l$ and $\a(\bn_1)(R_1)\cup\dots\cup\a(\bn_n)(R_1)\stackrel{o}{=}M$. Set $R_i=\a(\bn_i)(R_1)$. 
 
Let $\Gamma\subset\Zmm$ be the finite index subgroup of $\bm\in\Zmm$ such that $\a(\bm)(R_1)\stackrel{o}{=}R_1$. This is the decomposition and the finite index subgroup claimed in Theorem \ref{TMain1}.

To simplify notation, let us assume for the rest   of the paper that $\Gamma=\Zmm$ and  that $R_1$ has full measure. 
Under this assumption we can prove the following. 

\begin{lem}\label{weakmixing2}
For any hyperbolic element $\bn\in\Zmm$,

\begin{enumerate}[(i)]
\item $\an$ is Bernoulli and
\item  there is a set of full measure $R$ such that for any Weyl chamber $\cC$, if $x\in R$ then $$\bigcup_{z\in \W^u_\cC(x)\cap R}\W^s_\cC(z)$$ is a set of full measure. 
\end{enumerate}
\end{lem}
\begin{proof}
Let us show first that $\an$ is weakly mixing  (and hence Bernoulli) for any hyperbolic element $\bn\in\Zmm$. Using Pesin Ergodic decomposition Theorem we have a $k>0$ and a set $\hat R\subset M$ of positive $\mu$ measure invariant by $\a(k\bn)$ such that $\a(k\bn)|\hat R$ is weakly mixing  (and Bernoulli). This set is an ergodic component of  $\a(k\bn)$. 


 Now remember that  there is an element of the action that we previously denoted by $\bn_1$  such that $\a(\bn_1)$ is weakly mixing. Since it commutes with $\a(k\bn)$, it interchanges ergodic components of the latter map. If  there is more than one ergodic component for  $\a(k\bn)$, $\a(\bn_1)$ has a non-constant eigenfunction. Thus    weak mixing of  $\a(\bn_1)$ implies that  $\hat R\stackrel{o}{=}M$ and $k=1$, so that $\an$ is weakly mixing  and hence Bernoulli.

Let now $\cC$ be a Weyl chamber and let $\bn\in\cC$. Observe that if  $P$ is a  Pesin set for  a hyperbolic measure then for every  point 
$x\in P$ there is an open neighborhood $\mathcal P_x$ of a fixed size  (a Pesin box)  such that  for every  $y\in P\cap \mathcal P_x$ the local stable manifold of  $x$ intersects the local unstable manifold of $y$  transversally (at a single point). Since $\a(\bn)$ is hyperbolic and weak mixing, by the previous observation, for a.e. $x,y$ there is a non-negative  integer $k\geq 0$ such that $\w^u_\cC(\a(k\bn)(y))$ intersects transversally $\W^s_\cC(x)$ (simply take an iterate so that $\a(k\bn)(y)\in P\cap \mathcal P_x$). 
 Now let $k(x,y)$ be the minimum of such integers $k$. It is clear that $k(\an(x),\an(y))=k(x,y)$ for $\mu\times\mu$-a.e. $(x,y)$ and by the remark about Pesin boxes  $k(x,y)=0$ on a set of positive $\mu\times\mu$ measure. Now, since $\an$ is weak mixing, we have that $\an\times\an$ is ergodic and hence $k(x,y)=0$ a.e. 

Finally this statement is equivalent to the  statement (ii) of the lemma. 
\end{proof}

\subsection{Affine structures}\label{affinestructure}
In this subsection we shall define affine structures along the leafs of the invariant foliations and prove they are coherent, in the next subsection we shall see that holonomy maps are affine with respect  this affine parameters. 



Let $\chi$  be a  Lyapunov  exponent of $\alpha$ and $\mathcal W=\mathcal W^\chi$ be the corresponding Lyapunov foliation  defined $\mu$ almost everywhere. 
There is a unique $\a$-invariant  family of smooth affine parameters defined on almost every leaf of $\mathcal W$. Those affine structures   change continuously within any Pesin set, see \cite[Proposition 7.2]{KKRH} and hence  they can be defined  not only almost everywhere (at ``typical'' leaves  with respect to recurrence/ergodic behavior of $\a$)  but at other specific important places such as leaves  passing through periodic points that belong in a Pesin set. Those affine  parameters are obtained by integrating telescoping products, see the proof of \cite[Lemma 3.2]{KK}.  But at the same time affine parameters  define conditional measures  of $\mu$ with respect to $\mathcal W$. 

By \cite[Proposition 4.2]{KRH2} those affine structures are invariant with respect to  the holonomy  along  leaves 
of  the stable foliation of any  generic singular element  $\a(\bf t),\, \bf t\in\Rk$ for 
which $\chi(\bf t)=0$.  

We will present now the affine structures along stable manifolds of any Weyl chamber $\cC$ and prove their coherence. 


Since Lyapunov  hyperplanes are in general position  any combination of signs of Lyapunov exponents, except for all positive  or all negative, is possible for elements of the action $\a$,  and hence there are no resonances. In particular for every Lyapunov exponent $\chi$ there is a Weyl chamber $\cC_\chi$ such that inside  $\cC_\chi$, $\chi$ is the only positive  Lyapunov exponent. Hence the stable manifolds $\W^s_{\cC_\chi}(x)$ have codimension one. For any Weyl chamber $\cC$ the manifolds $\W^s_\cC(x)$ are the  intersections of  $\W^s_{\cC_\chi}(x)$
over those $\chi$ that  are negative in $\cC$. In particular, leaves of any Lyapunov foliation are intersection of $n-1$  codimension one stable manifolds. 

 Moreover, we can take elements of the action with pinched Lyapunov spectrum.               

  Let $\cC$ be a Weyl chamber and let $s=s(\cC)$ be the dimension of $\W^s_\cC$. Given $s\geq 1$, let $\mathcal D_s$ be the group of invertible diagonal matrices on $\mathbb R^s$ and let $Emb^{1+\epsilon}\left(\Rs, M\right)$ be the space of $C^{1+\epsilon}$ embeddings of $\Rs$ into $M$ with the topology of $C^{1+\epsilon}$ convergence on compact subsets, observe that in this way $Emb^{1+\epsilon}\left(\Rs, M\right)$ is a polish space. 
  
   From \cite[Proposition 2.7]{KRH1} we have that there is a  unique family of  smooth affine structures on the leaves of foliation  $\mathcal \w^s_{\cC}$.


\begin{prop}\label{affinestructures}
Let $\a$ be a $C^{1+\theta}$ action as in Theorem \ref{TMain1}. Then there are $\epsilon>0$, a set of full measure $R\subset M$ and a measurable map $H^\cC:R\to Emb^{1+\epsilon}\left(\Rs, M\right)$ such that denoting $H^\cC(x)=H_x$, 
\begin{enumerate}
\item  $H_x:\Rs\to \W^s_\cC(x)$, i.e.  $H_x(\Rs)=\W^s_\cC(x)$,
\item $H_x(0)=x$,
\item $D_0H_x:\Rs\to E^s_\cC(x)$ sends the standard basis into the frame of invariant spaces $E_{\chi_i}$ where $E^s_\cC(x)=E_{\chi_1}\oplus\dots\oplus E_{\chi_s}$ for some numeration of the Lyapunov exponents. 
\item There is a cocycle of diagonal maps of $\Rs$, $A:\Zmm\times R\to \mathcal D_s$ such that  $H_{\an(x)}\circ A(\bn,x)=\an\circ H_x$ for every $\bn\in\Zmm$ and $x\in R$. 
\end{enumerate}

Such a family is unique modulo composition with a diagonal map $D:R\to\mathcal D_s$, i.e. if $\hat H$ is another affine structure then for a.e. $x$, $H_x^{-1}\circ \hat H_x\in \mathcal D_s$.

 \end{prop}

In general, if $\a$ is $C^r$, the affine structures can be taken $C^{r-\delta}$ for any $\delta>0$. 
 
We want to prove coherence of the affine structures built in Proposition \ref{affinestructures} along stable manifolds.

\begin{prop}\label{coherence}
There is a set of full measure $R\subset M$ such that if $x,y\in R$ and $y\in \w^s_\cC(x)$ then $H_y^{-1}\circ H_x$ is an affine map with diagonal linear part. 
\end{prop}

%
%

\begin{proof}
Take some $\bn$ in the Weyl chamber $\cC$ such that $\w^s_\cC(x)$ is the stable manifold for $\an$. Let us number the Lyapunov exponents so that $\chi_1(\bn)<0,\dots\chi_s(\bn)<0$ for $\bn\in\cC$.  Take $L$ a Luzin set of continuity of $z\to H_z$ of $\mu$-measure close to $1$. Then there is a set of full measure $R_{\bn}\subset M$ such that whenever $x,y\in R_{\bn}$ then there are iterates $l_i\to +\infty$ such that $\a(l_i\bn)(x), \a(l_i\bn)(y)\in L$ for every $i$. This can be found using Birkhoff ergodic theorem as long as $L$ has $\mu$-measure larger that $1/2$. The set of full measure $R$ we claim in the proposition is the intersections of $R_{\bn}$'s for finitely many choices of $\bn\in\cC$ according to some pinching of the Lyapunov spectrum to be determined later.

Take $x,y\in R_{\bn}$ with $y\in  \w^s_\cC(x)$. Put $A^{(l)}(x)=A(l\bn,x)$ and similarly with $y$. We have that $$H_{\a(l\bn)(y)}^{-1}\circ H_{\a(l\bn)(x)}=A^{(l)}(y)H_y^{-1}\circ H_x\circ (A^{(l)}(x))^{-1}$$ and by continuity on Luzin sets and since $d(\a(l_i\bn)(x),\a(l_i\bn)(y))\to 0$  
\begin{eqnarray}\label{lim}
\lim_{l_i\to+\infty}\|H_{\a(l_i\bn)(y)}^{-1}\circ H_{\a(l_i\bn)(x)}-id\|_{C^1(B(1))}=0
\end{eqnarray} 
(here $\|\cdot\|_{C^1(B(1))}$ stands for the sup $C^1$ norm on the unit ball). 

We have also that $A^{(l)}(x)=:diag(\lambda_k^{(l)}(x))$ and we know that $$\lim_{l\to\pm\infty}|\lambda_k^{(l)}(x)|^{1/l}=\lim_{l\to\pm\infty}|\lambda_k^{(l)}(y)|^{1/l}=\exp(\chi_k(\bn))=:\lambda_k<1.$$
Set $P=H_y^{-1}\circ H_x$ then we have after (\ref{lim}) that

 \begin{eqnarray}\label{past}
 \lim_{l_i\to+\infty}\|A^{(l_i)}(y)P\circ(A^{(l_i)}(x))^{-1}-id\|_{C^1(B(1))}=0
\end{eqnarray}


Let $P=(P_1, \dots, P_s)$, take $P_k$ and let us show that $\partial_jP_k\equiv 0$ for $j\neq k$. For given $l$, we have that $$\partial_j \left(A^{(l)}(y)P\circ(A^{(l)}(x))^{-1}-id\right)_k=\frac{\lambda_k^{(l)}(y)}{\lambda_j^{(l)}(x)}(\partial_j P_k)\circ (A^{(l)}(x))^{-1}$$
if $j\neq k$.
Applying (\ref{past}) and that $\bigcup_{l_i} A^{(l_i)}(x)(B(1))=\Rs$,  we get that if $\lambda_k/\lambda_i>1$ then $\partial_j P_k\equiv 0$ on $\Rs$. Hence take $\bn_{k,j}\in\cC$ such that $\chi_k(\bn_{k,j})>\chi_i(\bn_{k,j})$ and we get that  $\partial_j P_k\equiv 0$ on $\Rs$ for $j\neq k$.

So, we have that $P_k(v)=P_k(v_k)$ only depends on the $k$th variable. Denote with $P'_k$ the derivative of $P_k$. Then again using formula (\ref{past}) and arguing as before we get that for any $R>0$, $$\lim_{l_i\to +\infty}\sup_{t\in B(R)}\left|\frac{\lambda_k^{(l_i)}(y)}{\lambda_k^{(l_i)}(x)}P'_k(t)-1\right|=0.$$
In particular this implies that $$\frac{\lambda_k^{(l_i)}(y)}{\lambda_k^{(l_i)}(x)}\to 1/P'_k(t)$$ for any $t\in\mathbb R$ which gives the claim since the left hand side does not depend on $t$ and then $P'_k$ is constant and hence $P_k$ is affine.

%

So we get that $P$ is an affine map with diagonal linear part. 

\end{proof}

Using uniqueness of affine structures and absolute continuity of the invariant measure we get that:

\begin{lem}\label{affinecoherence} Given any two Weyl chambers $\cC_1$ and $\cC_2$ such that the invariant foliations $\w^s_{\cC_1}\subset \W^s_{\cC_2}$, affine structures on the leaves of $\w^s_{\cC_1}$ are restrictions of affine structures on  $\W^s_{\cC_2}$. On the other hand, if $E^s_{\cC}=E^s_{\cC_1}\oplus\dots\oplus E^s_{\cC_l}$ we have that affine structures on $\w^s_{\cC}$ is a product (direct sum) of affine structures on $\w^s_{\cC_i}$, $i=1,\dots,l$.
\end{lem}

\begin{proof}
The main point in the Lemma is to prove that for a typical $x$, $\w^s_{\cC_1}(x)$ is a coordinate plane in the affine structure of $\W^s_{\cC_2}(x)$. Ones we settle this then the Lemma follows from uniqueness of affine structures. 

For a.e. point $x$ we may assume that (Lebesgue) almost every point $y$ in $\W^s_{\cC_2}(x)$ is a regular point, moreover we may assume also that for Lebesgue a.e. point $y\in\W^s_{\cC_2}(x)$, (Lebesgue) a.e. point $z\in\W^s_{\cC_1}(y)$ is a regular point. This is just absolute continuity of invariant foliations plus the fact that conditional measures are equivalent to Lebesgue. 

Let us denote with $H_z:\Rs\to \W^s_{\cC_2}(z)$ the affine structures on  $\W^s_{\cC_2}(z)= \W^s_{\cC_2}(x)$ based at $z$. 

Let $x$ be a point like in the previous paragraph and $y$ a regular point in $\W^s_{\cC_2}(x)$. By Proposition \ref{affinestructures}, item (\ref{linear}) we have that there is a coordinate plane $V$ such that $D_0H_x(V)=T_x\W^s_{\cC_1}(x)$

Let us consider the manifold $W=H_x^{-1}(\W^s_{\cC_1}(y))\subset \Rs$ and let us show that this manifold is a plane parallel to $V$. Let us assume that (Lebesgue) a.e. $z\in \W^s_{\cC_1}(y)$ is regular point. We shall show that for Lebesgue a.e. point $a\in W$, $T_aW=V$. By  Proposition \ref{affinestructures}, item (\ref{linear}) we have that $D_0H_z(V)=T_z\W^s_{\cC_1}(z)$, and since $\W^s_{\cC_1}(z)=\W^s_{\cC_1}(y)$ we get that $$T_z\W^s_{\cC_1}(y)=D_0H_z(V).$$

On the other hand, by Proposition \ref{coherence} we have that $H_x^{-1}\circ H_z$ is an affine map with diagonal linear part hence the derivative at $0$ of $H_x^{-1}\circ H_z$ is diagonal. So we have that if $a\in W$ and $H_x(a)=z$ is a regular point then \begin{eqnarray*}T_aW&=&T_aH_x^{-1}(\W^s_{\cC_1}(y))=D_zH_x^{-1}(T_z\W^s_{\cC_1}(y))=D_zH_x^{-1}(D_0H_z(V))\\&=&D_0(H_x^{-1}\circ H_z)(V).\end{eqnarray*}
Since $D_0(H_x^{-1}\circ H_z)$ is diagonal and $V$ is a coordinate plane we get that $$D_0(H_x^{-1}\circ H_z)(V)=V$$ and hence $T_aW=V$ for Lebesgue a.e. $a\in W$. Since $W$ is a $C^1$ manifolds then we have that $T_aW=V$ for every $a\in W$ and hence $W$ is a plane parallel to $V$ as wanted. 

\end{proof}

\begin{rem}Since not every point on a leaf of $\w^s_\cC$ is regular not all 
coordinate  lines correspond to actual leaves of Lyapunov foliations. But for a typical leaf this is true for almost every  coordinate line. On the other hand, in the setting of Lemma \ref{affinecoherence}, if $\w^s_{\cC_1}\subset \W^s_{\cC_2}$, then $\w^s_{\cC_1}$ uniquely extends to  a smooth foliation, indeed an affine foliation in the affine coordinates of $\W^s_{\cC_2}(x)$
\end{rem}
A corollary of the above is the following,
\begin{cor}
Let $\cC_{1}$ and $\cC_{2}$ be two Weyl chambers and let $\cC_3$ be the Weyl chamber such that $E^s_{\cC_1}\cap E^s_{\cC_2}=E^s_{\cC_3}$. Then there is a set of full measure $R\subset M$ such that if $x,y,z\in R$ with $z\in\w^s_{\cC_1}(x)\cap\w^s_{\cC_2}(y)$ then $\w^s_{\cC_3}(z)\subset \w^s_{\cC_1}(x)\cap\w^s_{\cC_2}(y)$ is a linear subspace in the affine structures along $\w^s_{\cC_1}(x)$ and $\w^s_{\cC_2}(y)$ tangent to the space corresponding to $E^s_{\cC_3}$.
\end{cor}

Now Hopf argument  can be applied and we obtain 
\begin{cor} Affine structures  on the leaves of the Lyapunov foliation $\mathcal W$ are  invariant with respect to the holonomy along the leaves of $\w_\cC$.
\end{cor}


\subsection{Uniformity of the holonomies} 

Now we want to show that holonomy maps along unstable manifolds between two stable manifolds are almost everywhere defined w.r.t. Lebesgue measure on stables. 
\begin{prop}\label{unifholonomy}
Let $\cC$ be a Weyl chamber. There is a set of full measure $R:=R_\cC\subset M$ such that if $x,y\in R$ and $y\in \W^u_\cC(x)$ then the holonomy along unstables $Hol^{u,\cC}_{x,y}:\W^s_\cC(x)\to \W^s_\cC(y)$ is defined for Lebesgue a.e. $z\in \W^s_\cC(x)$ and is affine, i.e. there is a diagonal linear map $B$ preserving the frame such that for a.e. $z\in  \W^s_\cC(x)$, there is a point $Hol^{u,\cC}_{x,y}(z)\in\W^u_\cC(z)\cap \W^s_\cC(y)$ and moreover $$Hol^{u,\cC}_{x,y}\circ H^\cC_x=H^\cC_y\circ B.$$
Lebesgue a.e. 
\end {prop}

\begin{proof}
Let as assume first that $\w^s_\cC$ is $1$-dimensional. Let $\cC_1$ and $\cC_2$ be Weyl chambers such that $E^u_{\cC_1}\oplus E^u_{\cC_2}=E^u_\cC$. By Lemma \ref{affinecoherence} we have that $\w^u_{\cC_1}$ and $\w^u_{\cC_2}$ are a pair of transverse linear sub-foliaitons of $\W^u_{\cC}$. Hence there is a set of full measure $R_0$ such if $x\in R_0$ and $y\in R_0$ then there are regular points $a, b\in \W^u_{\cC}(x)$ such that $a\in\w^u_{\cC_1}(x)$, $b\in\w^u_{\cC_1}(y)$ and $a\in\w^u_{\cC_2}(b)$. We have that $$Hol^{u,\cC}_{x,y}(z)=Hol^{u,\cC}_{b,y}\circ Hol^{u,\cC}_{a,b}\circ Hol^{u,\cC}_{x,a}(z),$$ see Figure 1. Since $E^s_{\cC}\oplus E^u_{\cC_1}=E^u_{\cC_3}$ for some Weyl chamber $\cC_3$ and $E^s_{\cC}\oplus E^u_{\cC_2}=E^u_{\cC_4}$  for some Weyl chamber $\cC_4$ we have by Lemma \ref{affinecoherence} that $\w^u_{\cC}$ and $\w^s_{\cC_1}$ are transverse linear subfoliations of $\w^u_{\cC_3}(x)$ and hence $Hol^{u,\cC}_{x,a}:\w^s_{\cC}(x)\to\w^s_{\cC}(a)$ is affine holonomy inside $\w^u_{\cC_3}(x)$ which is everywhere defined. Similarly with $Hol^{u,\cC}_{b,y}$ and $Hol^{u,\cC}_{a,b}$ which settles the claim when $\w^s_\cC$ is $1$-dimensional.

\begin{figure}[hbt]\label{prop271}
 \psfrag{s}[l]{\small $ \W^s_\cC(x)$}
  \psfrag{u}[l]{\small $ \W^u_\cC(x)$}
   \psfrag{d}[l]{\small $ \W^u_{\cC_1}$}
    \psfrag{e}[l]{\small $ \W^u_{\cC_2}$}
 \psfrag{f}[l]{\small $ \W^u_{\cC_4}$}
   \psfrag{a}[l]{\small $a$}
   \psfrag{b}[l]{\small $b$}
   \psfrag{x}[l]{\small $x$}
   \psfrag{y}[l]{\small $y$}
    \includegraphics[scale=0.7]{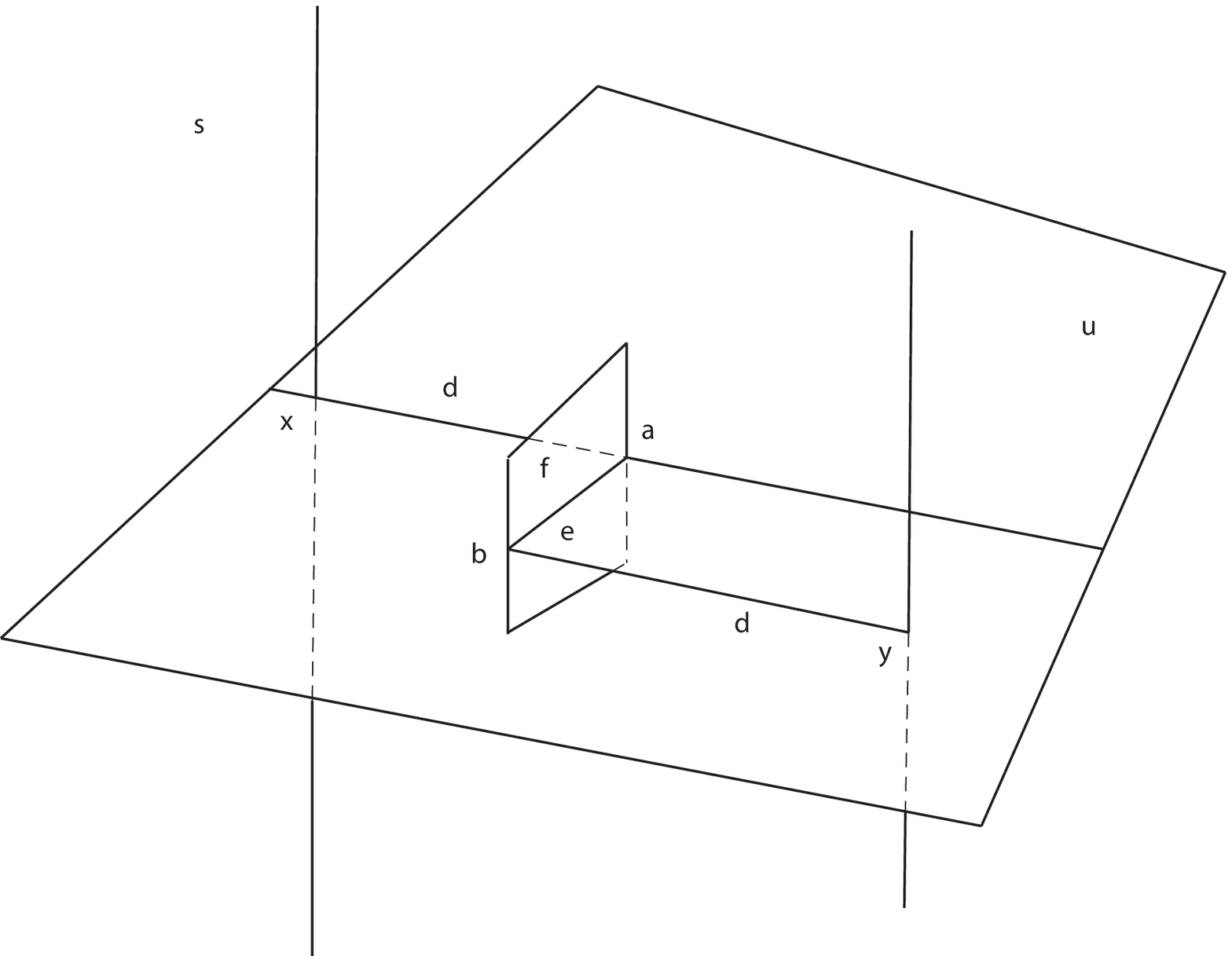}  
  \caption{Proposition 2.7.i}
   \end{figure}

Let us now assume that the dimension of  $\w^s_\cC$, $s(\cC)$, is larger that one and assume by induction that we have proven the Proposition for dimension $s<s(\cC)$. Then we have that there are Weyl chambers $\cC_1$ and $\cC_2$ such that $E^u_\cC=E^u_{\cC_1}\cap E^u_{\cC_2}$ and $E^s_\cC=E^s_{\cC_1}\oplus E^s_{\cC_2}$. We have that the dimensions $s(\cC_i)$, $i=1,2$ of $\w^s_{\cC_i}$ are strictly smaller than $s(\cC)$.

By the induction hypothesis, we have that for a.e. points $x$ and $y\in \w^u_{\cC}(x)\subset \w^u_{\cC_1}(x)$, $Hol^{u,\cC_1}_{x,y}:\W^s_{\cC_1}(x)\to \W^s_{\cC_1}(y)$ is everywhere defined and is affine. 

Taking $x, y$ typical points, we have that Lebesgue a.e. point $a\in \W^s_{\cC_1}(x)$ is regular, and $Hol^{u,\cC_1}_{x,y}(a)\in \W^u_{\cC_1}(a)\cap \W^s_{\cC_1}(y)$ is also regular, see Figure 2. Moreover, $Hol^{u,\cC_1}_{x,y}(a)=Hol^{u,\cC}_{x,y}(a)$. Indeed, $E^u_{\cC}\oplus E^s_{\cC_1}=E^u_{\cC_3}$ for some Weyl chamber $\cC_3$ which gives by Lemma \ref{affinecoherence} that $\w^u_\cC$ and $\w^s_{\cC_1}$ are a pair of transverse affine foliations in $\W^u_{\cC_3}(x)=\w^u_{\cC_3}(y)$ and hence $Hol^{u,\cC_1}_{x,y}(a)=Hol^{u,\cC}_{x,y}(a)\in \w^u_{\cC}(a)\cap\w^s_{\cC_1}(y)\subset\W^u_{\cC_1}(a)\cap \W^s_{\cC_1}(y)$, see Figure 2.

\begin{figure}[hbt]\label{prop272}
 \psfrag{s}[l]{\small $ \W^s_\cC(x)$}
  \psfrag{u}[l]{\small $ \W^u_\cC(x)$}
   \psfrag{d}[l]{\small $ \W^s_{\cC_1}$}
    \psfrag{e}[l]{\small $ \W^s_{\cC_2}$}
 \psfrag{f}[l]{\small $ \W^u_{\cC_4}$}
   \psfrag{a}[l]{\small $a$}
   \psfrag{b}[l]{\small $b$}
   \psfrag{x}[l]{\small $x$}
   \psfrag{y}[l]{\small $y$}
    \includegraphics[scale=0.7]{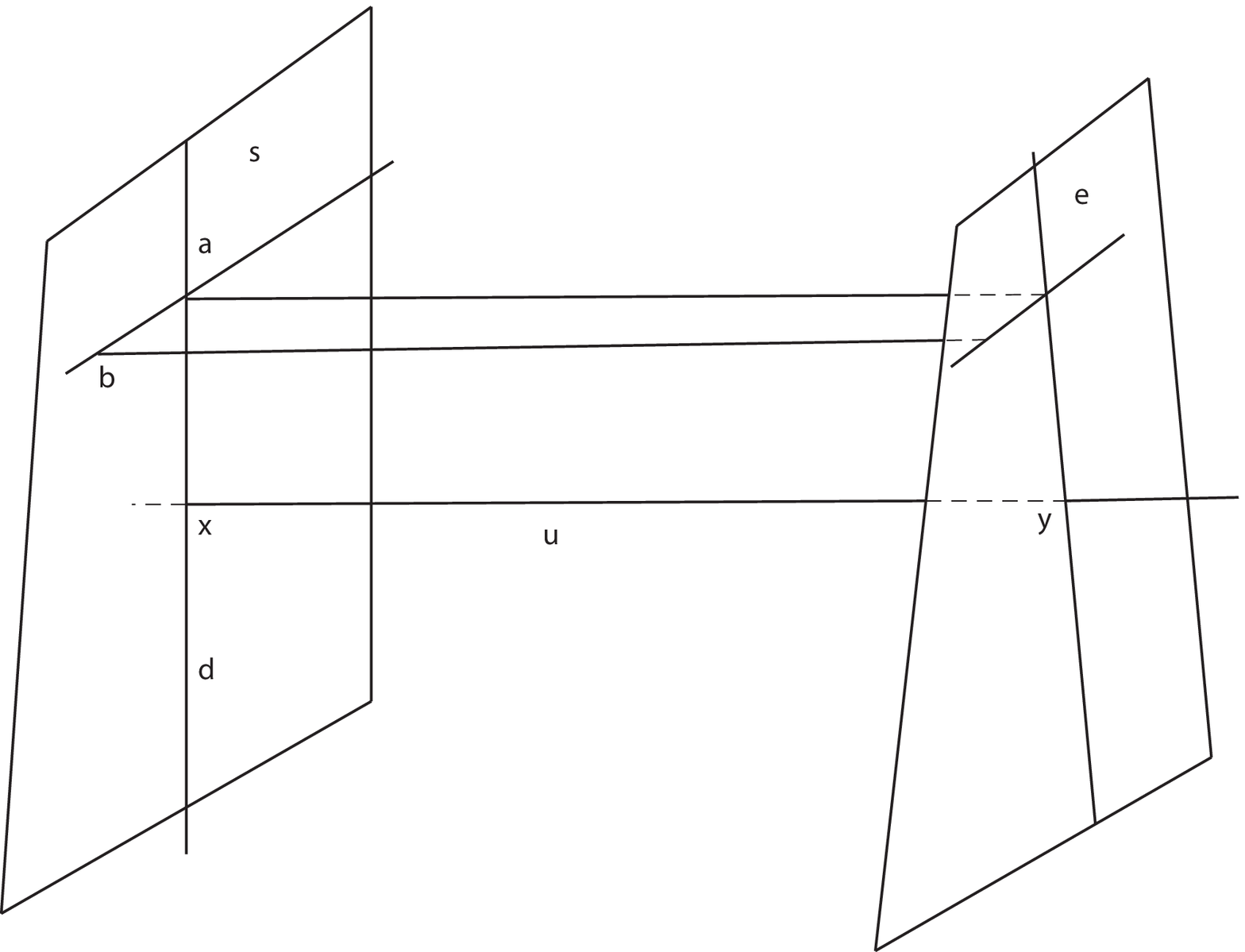}  
  \caption{Proposition 2.7.ii}
   \end{figure}

Now, for Lebesgue a.e. $z\in \w^s_{\cC}(x)$, $\w^s_{\cC_2}(z)\cap\w^s_{\cC_1}(x)$ intersects in a point $a\in \w^s_{\cC_1}(x)$ in the conditions of the previous paragraph. Hence $Hol^{u,\cC}_{x,y}(a)$ is well defined and again by induction we have that  $$Hol^{u,\cC_2}_{a,Hol^{u,\cC}_{x,y}(a)}:\W^s_{\cC_2}(a)\to \W^s_{\cC_1}(Hol^{u,\cC}_{x,y}(a))$$ is Lebesgue a.e. defined and is affine and  $$Hol^{u,\cC_2}_{a,Hol^{u,\cC}_{x,y}(a)}(z)\in \W^u_{\cC_2}(z)\cap \W^s_{\cC_2}(Hol^{u,\cC}_{x,y}(a)).$$ Moreover, since $E^u_{\cC}$ and $E^s_{\cC_2}$ are jointly integrable, arguing as before, we have that $$Hol^{u,\cC_2}_{a,Hol^{u,\cC}_{x,y}(a)}(z)=Hol^{u,\cC}_{x,y}(z).$$
Which gives that $Hol^{u,\cC}_{x,y}$ is Lebesgue a.e. defined and we get the Proposition.

\end{proof}

When $\cC$ is understood we shall note directly $Hol^u_{x,y}=Hol^{u,\cC}_{x,y}$.
\begin{prop}\label{prodstructure}
For $x,a,b\in R_\cC$, $a\in \W^u_\cC(x)$ and $b\in \W^s_\cC(x)$ we have that $$Hol^s_{x,b}(a)=Hol^u_{x,a}(b).$$
\end{prop}
\begin{proof}
The proof of this Proposition is very similar to the previous one. Assume without loss of generality that dimension of $\w^s_\cC$ is larger that $1$ (if not consider $\w^u_\cC$ instead). Even though we will use lot of Weyl chambers to settle the Proposition, all holonomies here will be w.r.t. invariant manifolds of the Weyl chamber $\cC$. 

\begin{figure}[hbt]
 \psfrag{s}[l]{\small $ \W^s_\cC(x)$}
  \psfrag{u}[l]{\small $ \W^u_\cC(x)$}
   \psfrag{w}[l]{\small $z_2$}
   \psfrag{t}[l]{\small $\bar{z_1}$}
   \psfrag{d}[l]{\small $\bar{z_2}$}
   \psfrag{z}[l]{\small $z_1$}
   \psfrag{a}[l]{\small $a$}
   \psfrag{b}[l]{\small $b$}
   \psfrag{x}[l]{\small $x$}
    \includegraphics[scale=0.7]{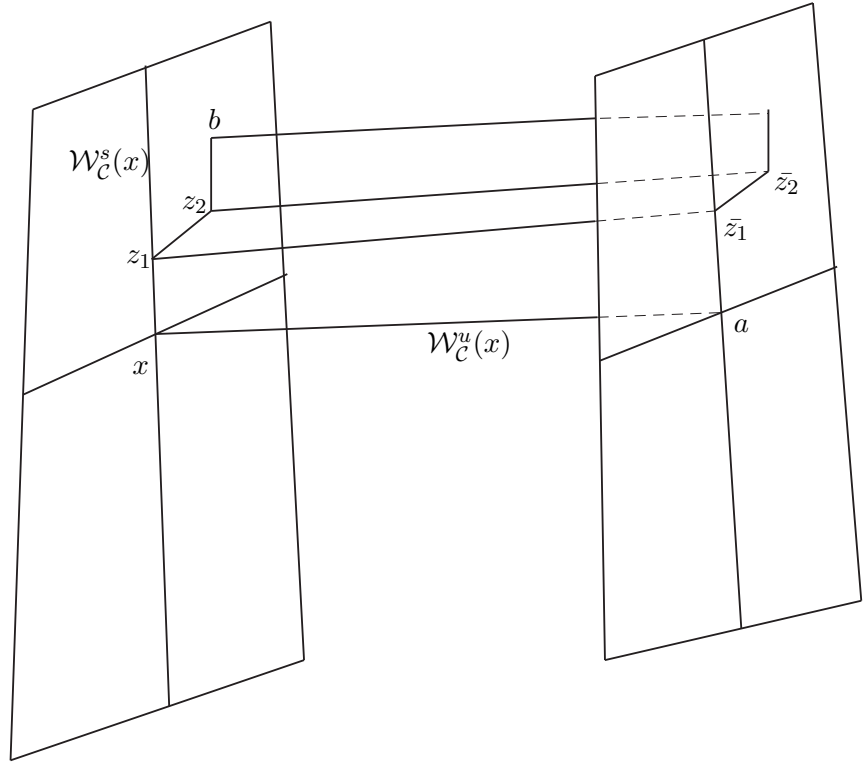}  
  \caption{Proposition 2.8}
   \end{figure}

Let $\cC_1$ and $\cC_2$ be Weyl chambers such that $E^s_{\cC_1}\oplus E^s_{\cC_2}=E^s_{\cC}$. Then we have that there are regular points $z_1,z_2\in \w^s_\cC(x)$ such that $\bar{z_1}:=Hol^u_{x,a}(z_1)$ and $\bar{z_2}:=Hol^u_{x,a}(z_2)$ are well defined and regular points and $z_1\in\w^s_{\cC_1}(x)$, $z_2\in \w^s_{\cC_2}(z_1)$ and $b\in \w^s_{\cC_1}(z_2)$, see Figure 3. We know on one hand that $Hol^s_{x,b}=Hol^s_{z_2,b}\circ Hol^s_{z_1,z_2}\circ Hol^s_{x,z_1}$. We have also that $Hol^u_{x,a}=Hol^u_{z_1,\bar{z_1}}=Hol^u_{z_2,\bar{z_2}}$.

Let $\cC_3$ and $\cC_4$ be Weyl chambers such that $E^u_{\cC}\oplus E^s_{\cC_1}=E^u_{\cC_3}$ and $E^u_{\cC}\oplus E^s_{\cC_2}=E^u_{\cC_4}$. Then we have, using that $x,a,z_1,\bar{z_1}\in \w^u_{\cC_3}(x)$ that $$\bar {z_1}=Hol^u_{x,a}(z_1)=Hol^s_{x,z_1}(a)$$ 
since all these holonomies take place inside $\w^u_{\cC_3}(x)$. Following this argument we get also that $z_1,z_2, \bar{z_1}, \bar{z_2}\in \w^u_{\cC_4}(z_1)$ and hence  $$\bar{z_2}=Hol^u_{x,a}(z_2)=Hol^u_{z_1,\bar{z_1}}(z_2)=Hol^s_{z_1,z_2}(\bar{z_1}).$$

Finally we get that, since $z_2,\bar{z_2},b,Hol^u_{z_2,\bar{z_2}}(b)\in\w^u_{\cC_3}(z_2)$, $$Hol^u_{x,a}(b)=Hol^u_{z_2,\bar{z_2}}(b)=Hol^s_{z_2,b}(\bar{z_2}).$$
Putting all together we get the Proposition.\end{proof}

\section{The arithmetic structure} 

\subsection{The quotient}\label{quotient}

Given $l\geq 1$, let $\mathcal D_l$ be the group of invertible diagonal matrices on $\mathbb R^l$ 
and let $\mathcal A_l$ be the group of affine maps on $\mathbb R^l$ whose linear term are in $\mathcal D_l$. In the sequel, when we say almost everywhere (a.e.) we mean w.r.t. Lebesgue measure unless another  measure is clearly specified.


Let us summarize in the following Proposition what we have proven in the previous sections. The first $3$ items are Proposition \ref{affinestructures}.


\begin{prop}\label{resume}
Given any Weyl chamber $\cC$ there is a set of full measure $R_ \cC$ such that if $x\in R_\cC$ then there are $C^r$ affine parameters $$H^\cC_x:\Rsc\to \W^s_\cC(x)$$such that 
\begin{enumerate}
\item $H^\cC_x(0)=x$ and  $x\to H^\cC_x\in Emb^r(\Rsc,M)$ is a measurable map. 
\item \label{normalization} $D_0H^\cC_x:\Rsc\to E^s_\cC(x)$ takes the standard frame in $\Rsc$ into the Lyapunov frame in $E^s_\cC(x)$. 

\item \label{linear} There is a cocycle $\hat{\alpha}^\cC:\Zmm\times R_\cC\to \mathcal D_{s(\cC)}$ such that for $\mu$-a.e. $x$, $$H^\cC_{\an(x)}\circ \hat{\alpha}^\cC_x(\mathbf{n})=\an\circ H^\cC_x.$$
\item \label{affineparameterscoherence} $ $[Proposition \ref{coherence}] If $x, y\in R_\cC$, $y\in \W^s_\cC(x)$ 
then there is $B_{x,y}\in \mathcal A_{s(\cC)}$ such that $$H^\cC_y=H^\cC_x\circ B_{x,y}$$
\item $ $[Lemma \ref{affinecoherence}] For any intersection of stable spaces of different Weyl Chambers, $\cC_1,\dots \cC_l$, $E=E^s_{\cC_1}\cap\dots\cap E^s_{\cC_l}$ the corresponding foliation $\W_E$ restricted to $\W^s_{\cC_i}(x)$ is a foliation by parallel planes in the affine parameters for every $i$. 
\item \label{changeorientation} [Proposition \ref{unifholonomy}] If $x,y\in R_\cC$ and $y\in \W^u_\cC(x)$ then the holonomy along unstables map $Hol^u_{x,y}:\W^s_\cC(x)\to \W^s_\cC(y)$ is defined for Lebesgue a.e. $z\in \W^s_\cC(x)$ and there is $D_{x,y}\in \mathcal D_{s(\cC)}$ such that $$Hol^u_{x,y}\circ H^\cC_x=H^\cC_y\circ D_{x,y}.$$
\item \label{holonomyproduct} [Proposition \ref{prodstructure}] For $x,a,b\in R_\cC$, $a\in \W^u_\cC(x)$ and $b\in \W^s_\cC(x)$ we have that $$Hol^s_{x,b}(a)=Hol^u_{x,a}(b).$$
\end{enumerate}
\end{prop}

Observe that in the affine coordinates the conditional measure is standard Haar measure with some normalization.

Let us fix by now a Weyl chamber $\cC$ that we will not be moved. Considering $-\cC$ we have affine parameters along the unstable foliation for points in $R_ {-\cC}$, take $R=R_\cC\cap R_{-\cC}$, we shall assume also that $R$ intersected with a.e. Lyapunov manifold has full Lebesgue measure, also we may need to reduce $R$ finitely many times to sets of still full measure satisfying adequate properties. Let us omit any reference to $\cC$ in the definitions above and denote with $H^s_x$ the affine parameters along stable manifolds $\w^s$ and $H^u_x$ for the affine parameters along unstable manifolds $\w^u$. 

We shall define a kind of covering or developing map for $M$, defined almost everywhere. Roughly, the idea is as follow, for a given $x$ we define $\hat {h}_x:\w^s(x)\times \w^u(x)\to M$, $\hat{h}_x(a,b)=Hol^u_{x,b}(a)=Hol^s_{x,a}(b)$ which is a point in $\w^u(a)\cap\w^s(b)$. Then we use affine parameters on $\w^s(x)$ and $\w^u(x)$ to define a map $h_x:\Rs\times\Ru\to M$. Since holonomies are only defined a.e. we need to take some care and that is what we do in the following paragraphs. 

Let us fix $x\in R$ and assume that $\W^s(x)\cap R$ and $\W^u(x)\cap R$ have full Lebesgue measure. Call $$R^s=(H^s_x)^{-1}(\W^s(x)\cap R)\subset \Rs$$ and $$R^u=(H^u_x)^{-1}(\W^u(x)\cap R)\subset \Ru.$$

Given $z^u\in R^u$, using item (\ref{changeorientation}) we can take $D^x_{z^u}\in \mathcal D_s$ such that $$Hol^u_{x,H^u_x(z^u)}\circ H^s_x=H^s_{H^u_x(z^u)}\circ D^x_{z^u}$$ (Remember that $Hol^u_{x,H^u_x(z^u)}:\W^s(x)\to \W^s(H^u_x(z^u))$

Let us define $h^u_x:\Rs\times R^u\to M$ by $$h^u_x(z^s,z^u)=H^s_{H^u_x(z^u)}(D^x_{z^u}(z^s)).$$

Observe that by item (\ref{changeorientation}) for a.e. $(z^s,z^u)$, 
\begin{eqnarray}\label{huviaholonomy}
h^u_x(z^s,z^u)=Hol^u_{x,H^u_x(z^u)}\left(H^s_x(z^s)\right).
\end{eqnarray}

Similarly, given $z^s\in R^s$, using item (\ref{changeorientation}) we can take $D^x_{z^s}\in \mathcal  D_u$ such that 
\begin{eqnarray}\label{defDs}Hol^s_{x,H^s_x(z^s)}\circ H^u_x=H^u_{H^s_x(z^s)}\circ D^x_{z^s}.\end{eqnarray}
(Remember that $Hol^s_{x,H^s_x(z^s)}:\W^u(x)\to \W^u(H^s_x(z^s))$.)

Let us define $h^s_x:R^s\times \Ru\to M$ by $$h^s_x(z^s,z^u)=H^u_{H^s_x(z^s)}(D^x_{z^s}(z^u)).$$

We also have here that for a.e. $(z^s,z^u)$, 
\begin{eqnarray}\label{hsviaholonomy}
h^s_x(z^s,z^u)=Hol^s_{x,H^s_x(z^s)}\left( H^u_x(z^u)\right).
\end{eqnarray}

\begin{lem}
For Lebesgue a.e. $(z^s,z^u)$, $h^s_x(z^s,z^u)=h^u_x(z^s,z^u)$.
\end{lem}
\begin{proof}
This is an immediate consequence of item (\ref{holonomyproduct}) and formulas (\ref{huviaholonomy}), (\ref{hsviaholonomy}).
\end{proof}

%
%
%
%
%
%
%


Let us denote $h_x=h^s_x=h^u_x$. 

\begin{lem}\label{affineonM}
For $\mu$ a.e. $x$ and for a.e. $(w^s,w^u)\in R^s\times R^u$ there is $L\in \mathcal A_m$ such that if we set $y=h_x(w^s,w^u)$ then 
\begin{enumerate}
\item $L(0,0)=(w^s, w^u)$ and
\item $h_x\circ L=h_y$ a.e.
\end{enumerate}

\end{lem}

\begin{proof}
Let $a=h_x(w^s,0)$. Let us prove first the proposition for $x$ and $a$. By item (\ref{affineparameterscoherence}) we have that there is $B\in \mathcal A_s$ such that $H_a^s=H_x^s\circ B$, $B(0)=w^s$. Hence by definition of $h^s$ we get that  
\begin{eqnarray*}
h^s_a(z^s,z^u)&=&H^u_{H^s_a(z^s)}(D^a_{z^s}(z^u))=H^u_{H^s_x(Bz^s)}(D^a_{z^s}(z^u))\\
&=&H^u_{H^s_x(Bz^s)}(D^x_{Bz^s}(Dz^u))=h^s_x(Bz^s,Dz^u).
\end{eqnarray*}
Where $D=(D^x_{Bz^s})^{-1}D^a_{z^s}$. We need to see that this $D$ does not depend on $z^s$.

From the definition of $D^x_{Bz^s}$, item (\ref{changeorientation}) and formula (\ref{defDs}) we know that 
$$Hol^s_{x,H^s_x(Bz^s)}\circ H^u_x=H^u_{H^s_x(Bz^s)}\circ D^x_{Bz^s}.$$

Hence $$(D^x_{Bz^s})^{-1}=(H^u_x)^{-1}\circ (Hol^s_{x,H^s_x(Bz^s)})^{-1}\circ H^u_{H^s_x(Bz^s)}$$
and similarly 

$$D^a_{z^s}= (H^u_{H^s_a(z^s)})^{-1}\circ Hol^s_{a,H^s_a(z^s)}\circ H^u_a.$$

Since $H^s_a(z^s)=H^s_x(Bz^s)$ and for any $b$, $Hol^s_{b,x}\circ Hol^s_{a,b}=Hol^s_{a,x}$ we get that 

\begin{eqnarray*}D=(D^x_{Bz^s})^{-1}D^a_{z^s}&=&(H^u_x)^{-1}\circ (Hol^s_{x,H^s_x(Bz^s)})^{-1}\circ Hol^s_{a,H^s_a(z^s)}\circ H^u_a\\&=&(H^u_x)^{-1}\circ Hol^s_{H^s_a(z^s),x}\circ Hol^s_{a,H^s_a(z^s)}\circ H^u_a\\&=&(H^u_x)^{-1}\circ Hol^s_{a,x}\circ  H^u_a = D_{a,x}.
\end{eqnarray*}

So, take $L=(B,D)\in\mathcal A_m$ in this case.

Now the general case of the Lemma follows from the observation that $y=h_x(w^s,w^u)=h_a(0,D^{-1}w^u)$ and applying the previous proof interchanging $u$ and $s$. \end{proof}

We have as an immediate corollary of the previous Lemma:
\begin{cor}\label{deck}
There is a set of full measure $R\subset \mathbb R^m$ such that if $(z^s,z^u), (w^s,w^u)\in R$ and $h_x(z^s,z^u)=h_x(w^s,w^u)$ then there is $L\in \mathcal A_m$ such that $h_x\circ L=h_x$ a.e.
\end{cor}

Let $\Gamma_x$ be the group of $L\in \mathcal A_m$ such that $$h_x(L(z^s,z^u))=h_x(z^s,z^u)$$ for Lebesgue  a.e. $(z^s,z^u)\in\Rs\times\Ru=\Rm$. $\Gamma_x$ should be thought as the group of deck transformations of the "covering" $h_x$. We call $\Gamma_x$ the {\it Homoclinic Group} since there is a correspondence between the points in $\W^u(x)\cap \W^s(x)$ and $\Gamma_x$. It is a nice experience for the reader to understand the previous construction in the case of a hyperbolic automorphism of $\mathbb T^2_{\pm}$


We consider $\Rm$ with its natural additive group structure and let $\lambda$ be Haar (=Lebesgue) measure on $\Rs\times\Ru=\mathbb R^m$. It is the product of Haar measure  on $\Rs$ and Haar on $\Ru$.
%
%
%

\begin{lem}\label{discont}
For Lebesgue a.e. $\bar z=(z^s,z^u)$ there is $c_x(\bar{z})>0$ and for any $\epsilon>0$ there is $\delta>0$ and a set $K_{\epsilon}(\bar z)\subset B_{\delta}(\bar z)$, such that 
\begin{enumerate}
\item $$\frac{\lambda(K_{\epsilon}(\bar z)\cap B_{\delta}(\bar z))}{\lambda(B_{\delta}(\bar z))}\geq 1-\epsilon$$
\item $h_x|_{K_{\epsilon}(\bar z)}$ is $1-1$
\item $\mu (h_x(A))=c_x(\bar{z})\lambda(A)$ for any measurable $A\subset K_{\epsilon}(\bar z)$. 

\end{enumerate}

\end{lem}

\begin{proof}
By Lemma \ref{affineonM} and Corollary \ref{deck}, it is enough to prove the Lemma when $\bar z=(0,0)$. Given $\epsilon>0$, since $x$ is a regular point it belongs to some Pesin set, and we may assume it is a density point of a Pesin set, hence we can take $\delta$ small so that for $K_\epsilon$ equal the pre-image by $h_x$ of the  Pesin set intersected with the ball of radius $\delta$ around $0$  items (1) and (2) hold. 

For item (3) notice that the conditional measure of $\mu$ along stables and unstables is Haar measure (with some normalization). Hence, by holonomy invariance of the conditional measures, we have that, locally on Pesin sets, the measure $\mu$ is the product of Haar on stable and Haar on unstable which gives Haar in affine coordinates $h_x$. \end{proof}
%
%


Applying item (\ref{linear}) of Proposition \ref{resume} both for $s$ and $u$, the definition of $h_x$ and Oseledcs Theorem we get the following:

\begin{lem}\label{quasconj}
For $\mu$-a..e. point $x$ we get,
\begin{enumerate}
\item There is a cocycle $\hat{\alpha}_x:\mathbb{Z}^{m-1}\times R\to \mathcal D_m$ such that $$h_{\an(x)}\circ\hat{\alpha}_x(\bn)= \an\circ h_x,$$ $R\subset \Rm$ is a full Lebesgue measure set. 
\item $$|\hat\alpha_x(k\bn)|^{1/k}\to D(\bn)=\diag(\exp\chi_1(\bn),\dots,\exp\chi_m(\bn)),$$ as $k\to\pm\infty$, where for a diagonal matrix $D$, $|D|$ the matrix with entries its absolute value and $|D|^{1/k}$ is its real positive $k$-th root. 

\item  For any $\bn\in\mathbb{Z}^{m-1}$, $$\hat{\alpha}_x(\bn)\Gamma_x=\Gamma_{\an(x)}\hat{\alpha}_x(\bn).$$
\end{enumerate}
\end{lem}
The last assertion follows from the definition of $\Gamma_x$ and Corollary \ref{deck}.

\begin{prop}\label{the infratorus}
For $\mu$-a.e. point $x\in M$, $\Gamma_x$ contains a normal subgroup of finite index isomorphic to $\Zm$ acting by translations on $\Rm$.
\end{prop}
\begin{proof}
Given $x$ let $Tr_x\subset\Gamma_x$ be the normal subgroup of translations in $\Gamma_x$. We always regard $\Rm$ with the standard inner product. Let $E(x)\subset\Rm$ be the vector space generated by the translations in $Tr_x$. By Lemma \ref{discont} we have that $Tr_x$ is discrete and hence the quotient $E(x)/Tr_x$ is a torus. Let $v(x)$ be the volume of $E(x)/Tr_x$. Notice that $y\to v(y)$ is a measurable map. Indeed, by Lemma \ref{affineonM} we have that for a.e. $y$ there is a $L_{x,y}$ such that $h_x\circ L_{x,y}=h_y$, and by the construction we get that we can choose $L_{x,y}$ in such a way that $x\to L_{x,y}$ is measurable. Let $D_{x,y}$ be the linear part of $L_{x,y}$. Some linear algebra gives that $D_{x,y}Tr_y=Tr_x$ which gives the measurability of $y\to v(y)$. 

On the other hand, by Lemma \ref{quasconj} and arguing as before we have that for any $\bn\in\mathbb{Z}^{m-1}$ and a.e. $x\in M$, $\hat{\alpha}_x(\bn)Tr_x=Tr_{\an(x)}$. Let $D(\bn)=\diag(\exp\chi_1(\bn),\dots,\exp\chi_m(\bn))$.

Let $y$ be a typical point. Let $d=\dim E(y)$. It is clear from the previous analysis and ergodicity of $\alpha$ that $d$ does not depend on $y$. Let us assume by contradiction that $0<d<m$ then, since the Lyapunov exponents of $\a$ are in general position we can chose an element $\bn\in\Zmm$ such that for $l\to+\infty$ the action of the iterates $D(l\bn)$ in the $d$th exterior product $\Lambda^d(\Rm)$ expands the volume element of $E(y)$ exponentially. Since the cocycle $\hat\a_y(l\bn)$ is asymptotically $D(l\bn)$ we have that  $\hat\a(l\bn)_y$ also expands the volume element of $E(y)$ exponentially. Hence we get that for $l\to+\infty$, $v(\a(l\bn)(y))$ tends to infinity contradicting that by recurrence $\a(l\bn)(y)$ has to return to a region where $v(y)$ is finite. 

Hence we get that either $d=0$ or $d=m$. 

If $d=0$ then $\Gamma_x$ has no translation part and hence is abelian (consider the homomorphism $D_0:\Gamma_x\to\mathcal D_m$, $D_0L=$ derivative of $L$ at $0$). So, $\Gamma_x$ is conjugated by a translation to the action of a diagonal subgroup on $\Rm$. 

Let $F(x)=Fix(\Gamma_x)$ be the set of points fixed by all the elements in $\Gamma_x$. Observe that we have that $F(x)$ is an affine subspace parallel to the coordinate axes. We shall show that $0\in F(x)$ for $\mu$-a.e. $x$ (i.e. $F(x)$ is a linear subspace) and then reach a contradiction.

From Lemma \ref{quasconj} we get that $F(\an(x))=\hat{\alpha}_x(\bn)F(x)$. Since $\hat{\alpha}_x(\bn)$ is diagonal and by ergodicity of $\an$ we get that there is a linear subspace $F$ paralell to the axes, independent of $x$ and a unique vector $p(x)$ perpendicular to $F$ such that $F(x)=p(x)+F$ for $\mu$-a.e. $x$. Moreover $p(x)$ is measurable, $\hat{\alpha}_x(\bn)F=F$ and $\hat{\alpha}_x(\bn)p(x)=p(\an(x))$ for every $\bn\in\mathbb{Z}^{m-1}$. Since by Lemma \ref{quasconj} the cocycle $\hat{\alpha}_x(\bn)$ is diagonal and asymptotically $D(\bn)=\diag(\exp\chi_1(\bn),\dots,\exp\chi_m(\bn))$ we get, working with each coordinate of $p(x)$ at a time that $p(x)=0$. 

So we have that $F(x)=F$ is a linear subspace independent of $x$ and $0\in F=F(x)$. In particular $\Gamma_x\subset \mathcal D_m$. Since $h_x|\Rs\times\{0\}$ coincides with the affine parameter along the stable manifold of $x$ we get that $\Rs\times\{0\}\subset F$, similarly we get that $\{0\}\times\Ru\subset F$ and hence $F=\Rm$, i.e. $\Gamma_x=\{id\}$ is trivial.

In particular, by Corollary \ref{deck} we get that $h_x:\Rm\to M$ is 1-1 Lebesgue a.e. Let $\nu=(h_x)_*\mu$. By Lemma \ref{discont}, $\nu$ is a a probability measure equivalent to Lebesgue measure and invariant by the action $\alpha_0(\bn):=h_x^{-1}\circ\an\circ h_x$. By Lemmas \ref{affineonM} and \ref{quasconj} we have that $\alpha_0(\bn)$ is affine for every $\bn$. But this is a contradiction since affine maps on $\Rm$ do not admit positive entropy invariant probability measures but $(\alpha_0(\n\bn),\nu)$ is measurably isomorphic through $h_x$ to $(\an,\mu)$.

So we get that $d=m$ and hence $E(x)=\Rm$. Recall that from Lemma \ref{discont} we know that $Tr_x$ is discrete. Let us take a linear map and conjugate $Tr_x$ to $\Zm$ and $\Gamma_x$ to $\Gamma$. Since $Tr_x$ is normal in $\Gamma_x$ then we have that $\Zm$ is normal in $\Gamma$. Hence we have that $\hat\Gamma=\Gamma/\Zm$ is identified with a subgroup of affine maps on the torus $\Tm=\Rm/\Zm$, and $\Rm/\Gamma_x\sim\Rm/\Gamma=\Tm/\hat\Gamma$. Again, using Lemma \ref{discont} we have that $\hat\Gamma$ cannot have any recurrence and hence it has to be finite, finishing the proof.
\end{proof}

So we get that $\Rm/\Gamma_x$ is a well defined orbifold. By Corollary \ref{deck} we get that $h_x:\Rm/\Gamma_x\to M$ is an isomorphism.

Let $\a_0:\mathbb{Z}^{m-1}\to Aff(\Rm/\Gamma_x)$ be the abelian action defined by conjugating $\an$ with $h_x$
\begin{eqnarray}\label{conjugacy}
h_x\circ\a_0(\bn)=\an\circ h_x
\end{eqnarray}
for any $\bn\in\Zmm$. Let $\nu=(h_x)_*\mu$ be the pullback measure. 
\begin{cor}\label{orbifold}
$\Gamma_x$ is isomorphic either to $\Zm$ or to $\Zm\ltimes\{\pm id\}$, $\nu=\lambda$ is Haar measure (or projected Haar measure) on $L:= \Rm/\Gamma_x$ and $h:=h_x$ is a measurable conjugacy between $(\a_0,\lambda)$ and $(\a,\mu)$.
\end{cor}
\begin{proof}
The only thing that needs proof in this Corollary is the property on the group and on $\nu$. We know already that $\Zm$ is a finite index normal subgroup of $\Gamma_x$. Let $\tilde\a_0$ be the  lifting of the action $\a_0$ to the finite covering $\Tm$ and let us lift also the measure $\nu$ to $\Tm$. By the generic position of the Lyapunov exponents for $\a$ we get that $\tilde\a_0$ is a restriction of a maximal Cartan action to a finite index subgroup and hence we get that the lifted measure is absolutely continuous w.r.t. Lebesgue and invariant and hence is Haar measure. Hence we get the claim on the measure. Again using that $\tilde\a_0$ is a maximal Cartan action on $\Tm$ we get that the only possibility for $\Gamma_x/\Zm$ is to be $\{\pm id\}$ and we get the Corollary. 
\end{proof}

\begin{proof}[ \bf Proof of Theorem \ref{TMain1} ]
By the weak-mixing reduction subsection \ref{weakmixingreduction} we have the a set $R_1$ with $\mu(R_1)>0$ and a finite index subgroup stabilizing $R_1$. By restricting the action to this finite index subgroup and normalizing the measure we may assume that the measure is weak-mixing and hence by Lemma \ref{weakmixing2} we get that there is a set of full measure $R_2$ such that for any $x\in R_2$,  $$R_3:=\bigcup_{z\in \W^u_\cC(x)\cap R}\W^s_\cC(z)$$ is a set of full measure. 

As a consequence of the construction of $h_x$ we get that the image of $h_x$ contains $R_3$ and hence has full measure, hence $h_x:(\Rm/\Gamma_x,\nu)\to (M,\mu)$ is an isomorphism conjugating $\alpha$ with $\alpha_0$. By Corollary \ref{orbifold} we have that $\Rm/\Gamma_x$ is either a torus or the infratorus $\Tm_{\pm}$. 

Take some $x$ and define $h=h_x^{-1}$. This gives the first part and items (\ref{TM1}) and (\ref{TM2}) of Theorem \ref{TMain1}.  Item (\ref{TM3}) follows from construction, i.e. $h^{-1}$ restricted to the affine spaces parallel to the axes is affine parameters of corresponding stable manifold. 

Finally item (\ref{TM4}) is a consequence the Journ\'e Theorem, see \cite[Theorem 5.7 and Proposition 5.13]{dlL} and item (\ref{TM3}). More precisely, consider a Pesin set $\Lambda$ of large measure and take the set $W^u_{loc}(\Lambda)\cup W^s_{loc}(\Lambda)$, where $W^u_{loc}(\Lambda)=\bigcup_{y\in\Lambda}W^u_{loc}(y)$. Restrict $h$ to $\Sigma$. Since affine structures and holonomies varies continuously on $\Lambda$ we have that $h$ is continuous on $\Sigma$. Moreover, since stable foliations are H\"older continuous along $W^s(\Lambda)$ we have that derivatives of $h$ along the stable direction are H\"older. Similarly for $W^u(\Lambda)$ and along the unstable foliation.  Finally we get by Journ\'e's Theorem that $h$ is smooth in the Whitney sense on $W^s(\Lambda)\cap W^u(\Lambda)\supset\Lambda$. \end{proof}

%
%
%
%
%
%
%
%
%
%
%
%
Observe that we can use $h_x$ and its restriction to planes parallels to the axes as new affine parameters. This are still smooth parameters and with this new affine parameters holonomies are isometries. 

For future use, let us summarize some properties of the measurable conjugacy.
\begin{lem}\label{smoothconjugacy}
There is an $\a_0$-invariant set of full Lebesgue measure $R\subset \mathbb R^m/\Gamma_x$ in the infratorus such that for every $v\in R$ the measurable conjugacy $h_x$ restricted to any invariant linear subspace $v+E\subset  \mathbb R^{m}/\Gamma_x$ trough $v$ coincides with the affine structure through $\w_E(h_x(v))\subset M$ ($\w_E(y)\subset M$ is the invariant manifold associated to $E$ through $y$). In particular, for a.e. $y$ and every Weyl chamber $\cC$, $h_x^{-1}|\w^u_\cC(y)$ is a diffeomorphism onto $h_x^{-1}(y)+E^u_\cC$ (the corresponding unstable plane) and holonomies are isometries in this affine parameters. 
\end{lem}

\section{Anosov actions}

An action $\a:\Zk\to Diff(M)$ is an Anosov action if there is $\mathbf{n_0}\in\Zk$ such that $\a(\mathbf{n_0})$ is an Anosov diffeomorphisms. 

\begin{thm}
Let $(\a,\mu)$ be an action as in Theorem \ref{TMain1}, i.e. a maximal rank action, assume furthermore that $\a$ is an Anosov action. Then $\alpha$ is smoothly conjugated to $\alpha_0$ and hence $M$ is indeed diffeomorphic to a (standard) torus. 
\end{thm}



We shall prove that the measurable conjugacy in Theorem \ref{TMain1} is indeed a homeomorphism. 

\begin{proof}

Let $x\in M$ be a regular point and consider $h_x:\Rs\times\Ru\to M$ which is defined almost everywhere with respect to Lebesgue measure on $\Rs\times\Ru$, we consider here the Anosov element and take the Weyl chamber containing this Anosov element for the definition of $h_x$. First of all observe that by definition we get that there is $\epsilon>0$ and $\delta>0$ small such that if $(z^s,z^u)$ is $\delta$ close to $(0,0)$ then $$h_x(z^s,z^u)=W^s_{\epsilon}(h_x(z^s,0))\cap W^u_{\epsilon}(h_x(0,z^u)).$$
This implies that $h_x$ restricted to the $\delta$ neighborhood $B_{\delta}(0,0)$ of $(0,0)$ is continuous. Now, using Proposition \ref{affineonM} we get that for Lebesgue a.e. $(w^s,w^u)$ there is an isometry $L$ such that if $y=h_x(w^s,w^u)$ then $L(0,0)=(w^s, w^u)$ and $h_x\circ L=h_y$ a.e. In particular $h_x$ restricted to the $\delta$ neighborhood of $(w^s,w^u)$, $B_{\delta}(w^s,w^u)$ is also continuous since $h_y$ is continuous when restricted to the $\delta$ neighborhood $B_{\delta}(0,0)$ of $(0,0)$ and $h_x=h_y\circ L^{-1}$ and $L$ is an isometry. Since $\delta$ is fixed we get that the union of the $\delta$ balls around Lebesgue a.e. point is $\Rs\times\Ru$ and hence $h_x$ is continuous everywhere. 

Following the same reasoning as in the proof of Theorem \ref{TMain1} we get that $h_x$ is indeed a covering map and taking the quotient by the group of deck transformations we get that $h_x$ is a homeomorphisms and a conjugacy between the affine action $\a_0$ on an infratorus and the action $\a$.

Observe that here the infratorus is a manifold, hence, applying the results in \cite{rhglobal} or \cite{RH-W} on global rigidity of maximal  Anosov rank actions, we get the smooth classification. \end{proof}



\section{Proof of Theorem \ref{TMain2}}
From Theorem \ref{TMain1} we have a decomposition into weak mixing components, a corresponding finite index subgroup of $\mathbb{Z}^{m-1}$ and a measurable conjugacy $h:(M,\nu)\to (L,\lambda)$ between $\a$ and an affine action $\a_0$ when restricted to this finite index subgroup. Here we shall show how $h$ coincides with a continuous onto map from an $\a$-invariant open set $O$ and $L\setminus F$ for some finite $\a_0$-invariant set $F$ satisfying the conclusion of Theorem \ref{TMain2}.

The first step is to identify the open set $O$ and the finite set $F$. Given a Weyl chamber $\cC$ and a regular point $x$ let $\W^\sigma_\cC(x)$, $\sigma=s,u$ be the stable and unstable manifolds through $x$ corresponding to this Weyl chamber. With a subscript  $W^\sigma_{\cC,loc}(x)$ we shall denote the local invariant manifold once a Pesin set is understood. Let $\cC_1,\dots, \cC_{m}$ denote the Weyl chambers with only one positive exponent. 

We say that a closed set $B\subset M$ is a {\it box} or a {\it cube} if it is homeomorphic to the unit cube in $\Rm$ and its boundary $\partial B$ is in the union of stable and unstable manifolds for different Weyl chambers, i.e. there are regular points $x^{i,\pm},  i=1,\dots, m$, such that
$$\partial B\subset \bigcup_{i=1}^{m}\left(\w^s_{\cC_i}(x^{i,-})\cup \w^s_{\cC_i}(x^{i,+})\right).$$ 

We shall call each piece $$\partial^{\pm}_{\cC_i} B_l:=\partial B_l\cap \w^s_{\cC_i}(x^{i,\pm})$$ a {\it face} of the cube $B$ (or of its boundary $\partial B$). We are assuming that $x_i^+$ and $x_i^-$ do not belong to the same stable manifold, if not take connected components. 

Given a Pesin set $P$, if we can take $x^{i,\pm}_l\in P$ close enough so that $$\partial B\subset \bigcup_{i=1}^{m}\left(W^s_{\cC_i,loc}(x^{i,-})\cup W^s_{\cC_i,loc}(x^{i,+})\right),$$ then we say that $B$ is a {\it good box} and we get as a consequence that 
$$\partial^{\pm}_{\cC_i} B=\partial B\cap W^s_{\cC_i,loc}(x^{i,\pm})$$

\begin{lem}\label{boxes}
For any given Pesin set $P$ and for $\nu$ a.e. point $x\in P$ there is a sequence of good boxes $B_l$, $l\geq 1$, such that:
\begin{enumerate}
\item\label{shrink} $x\in B_l\subset  int  B_{l-1}$ and $\cap_{l\geq 1}B_l=\{x\}$,
\item\label{cube} $B_l$ is diffeomorphic to the closed unit cube,
\item\label{separatix} Each separatrix of $W^u_{\cC_i,loc}(x)\setminus\{x\}$ intersects a corresponding face of $\partial^{\pm}_{\cC_i} B_l\neq \emptyset$,
\item\label{conjugacy1} $h$ is defined a.e. w.r.t. Lebesgue measure on $\partial B_l$  and coincides with a diffeomorphism with $C^r$ norm bounded by a constant depending only on $P$ and $h(\partial B_l)$ is the boundary of a linear cube $\hat B_l$,
\item\label{conjugacy2} For $i=1,\dots m$, $W^{s}_{\cC_i,loc}(x)$ disconnects $B_l$ into two connected components named $B_{i,l}^{\pm}$ which are also boxes and $h(\partial B_{i,l}^{\pm})$ is the boundary of a corresponding linear cube $\hat B_{i,l}^{\pm}$.
\end{enumerate}
Moreover, the points $x^{i,\pm}_l\in P$ can be further required to belong to a given full measure set (e.g. has a dense orbit in the support of $\nu$, etc). 
\end{lem}

\begin{proof}
Consider $x$ a density point on the Pesin set $P$ intersected with the set of full measure in Lemma  \ref{smoothconjugacy}. Since $W^{s}_{\cC_i,loc}(x)$ locally separates a neighborhood of $x$ in two connected components, we can take the points $x^{i,\pm}_l$ from the same set as $x$ and from both sides of $W^{s}_{\cC_i,loc}(x)$, approaching $x$. 

Parts (\ref{shrink}), (\ref{cube}) and (\ref{separatix}) follow from uniformity of foliations on Pesin sets. Parts (\ref{conjugacy1}) and (\ref{conjugacy2}) are a consequence Lemma \ref{smoothconjugacy} and uniformity on Pesin sets (Luzin set).
\end{proof}

Given a good box $B$, for $\sigma=s,u$ let $$W^{\sigma}_{\cC_i,B}(x)=B\cap W^{\sigma}_{\cC_i,loc}(x)$$ be the core of the box $B$. Let $\W^{(u,\pm)}_{\cC_i}(x)$ be the separatrix of $\W^u_{\cC_i}(x)\setminus\{x\}$ that intersects $\partial^{\pm}_{\cC_i} B$ and for a regular point $y\in \W^s_{\cC_i}(x)$ we define $\W^{u,\pm}_{\cC_i}(y)$ accordingly. Finally, for $r>0$ let $W^{u,\pm}_{\cC_i,r}(y)$ be the segment of length $r$ with respect to the affine parameters given by $h$ (see Lemma \ref{smoothconjugacy}) inside $\W^{u,\pm}_{\cC_i}(y)$. 

Let us fix $x$ a point as in Lemma \ref{boxes}, and $l\geq 1$, we shall omit the subscript $l$ in $B_l$ in the sequel. Define $$O=\bigcup_{\bn\in \Gamma} \an(int B).$$ We have that the corresponding set $$\bigcup_{\bn\in \Gamma} \a_0(\bn)(int \hat B)$$ is an open nonempty $\ao$-invariant set, then by Berend's Theorem \cite{dberend} we get that it is the complement to a finite $\ao$-invariant set $F$. Observe that singular points of the infratorus are contained in $F$ since points in $L\setminus F$ have a cube neighborhood. We may also assume that $$O=\bigcup_{\bn\in \Gamma} \an(B)\;\;\;\mbox{and}\;\;\; L\setminus F=\bigcup_{\bn\in \Gamma} \a_0(\bn)(\hat B)$$
because the faces of the boundary of $B$ (respectively of $\hat B$) is formed by stable manifolds of different elements of the action passing trough points which can be taken to have dense orbit on the support of the measure and hence each face of the boundary is mapped eventually completely inside $int B$ (respectively $int \hat B$).


For a point $x$ as in Lemma \ref{boxes}, let $r^{i,\pm}$ be the length of the separatrix $W^{u}_{\cC_i,B}(x)\cap B^{\pm}=W^{u}_{\cC_i,B}(x)\cap\W^{(u,\pm)}_{\cC_i}(x)$  measured with respect to the affine parameter in $\W^{u}_{\cC_i}(x)$ (i.e. $W^{u}_{\cC_i,B}(x)\cap B^{\pm}=W^{u,\pm}_{\cC_i,r^{i,\pm}}(x)$).

\begin{lem}\label{intersection}
For $1\leq i\leq m$ and for $\nu$ a.e. $x$ and for any full Lebesgue measure subset $R\subset W^{s}_{\cC_i,B}(x)$, $$\bigcup_{z\in R} W^{u,\pm}_{\cC_i,r^{i,\pm}}(z)\stackrel{o}{=}B^{\pm}$$ w.r.t. $\nu$-measure. In particular, for $\nu$ a.e. point in $y\in B$, $$W^u_{\cC_i, B}(y)\pitchfork W^s_{\cC_i,B}(x)\neq\emptyset.$$
\end{lem}
\begin{proof}
The assertion on the transverse intersection is an immediate consequence of the first assertion. The first assertion is an immediate consequence of Lemma \ref{smoothconjugacy}, that $h$ is a measurable conjugacy between $\nu$ and $\lambda$ and that the same assertion for the linear case is trivial.
\end{proof}

Let $E^\sigma_{\cC_j}$, $\sigma=s,u$, $j=1,\dots m$ be the corresponding stable and unstable invariant spaces for the linear action. Let us use the same notation for their projection on the infratorus $L$. Observe that as long as $z+E^\sigma_{\cC_j}\subset \Rm$ does not contain a point corresponding to a singular point of the infratorus, the natural projection $p$ from $z+E^\sigma_{\cC_j}$ into $L$ is one to one and onto the corresponding affine space $E^\sigma_{\cC_j}(p(z))$. 


Given a box $\hat B$ as in Lemma \ref{boxes} and $\hat y\in \hat B$, let $E^\sigma_{\cC_j,\hat B}(\hat y)$ be the connected component of $E^\sigma_{\cC_j}(\hat y)\cap \hat B$ containing $\hat y$. Given a regular point $y\in B$ recall that $W^{\sigma}_{\cC^i,B}(y)$ is the connected component of $B\cap \W^{\sigma}_{\cC_i}(y)$ containing $y$. 

\begin{lem}\label{disconnect}
For $\nu$ a.e. point $y\in B$, $W^s_{\cC_i,B}(y)$ is a $k$-dimensional box and $$h(W^s_{\cC_i,B}(y))=E^s_{\cC_i,\hat B}(h(y)).$$ Moreover, for $\nu$ a.e. $y, z\in B$ with $z\in W^{u}_{\cC_i,B}(y)$, $Hol^s_{y,z}:\W^{s}_{\cC_i}(y)\to\W^{s}_{\cC_i}(z)$ is such that $$Hol^s_{y,z}(W^{s}_{\cC_i,B}(y))=W^{s}_{\cC_i,B}(z).$$ Finally, $W^{s}_{\cC_i,B}(y)$ disconnects $B$ in two connected components, homeomorphic to boxes. 
\end{lem}

\begin{proof}
The first assertion follows from Lemma \ref{smoothconjugacy} and the constructions of the boxes in Lemma \ref{intersection}. The second is a direct consequence of the first and the same property for the linear case. 

Finally, let us prove the third assertion. From Lemma \ref{boxes} and the first part we get that $h(\partial B_l\cup W^{s}_{\cC_i,B}(y))=\partial\hat B\cup E^s_{\cC_i,\hat B}(h(y))$ and on this domain $h$ is a diffeomorphism by Lemma \ref{smoothconjugacy}.

Taking $B$ small enough so that it is in a neighborhood chart and using  Sch\"onflies Theorem \cite{jalexander, mbrown, mmorse, bmazur} we get that the pair $(B, W^{s}_{\cC_i,B}(y))$ is homeomorphic to the pair $(I^m,I^{m-1}\times\{1/2\})$. Indeed by the H-cobordism theorem, it follows that it is diffeomorphic if $m-1\neq 3$, (i.e. $m\neq 4$).
\end{proof}


\begin{lem}\label{boxeverypoint}
Given a set $R$ of full measure, for {\it every} $y\in B$ there is a sequence of boxes $y\in B_{n+1}(y)\subset int B_n(y)$, $\geq 1$, such that 
\begin{enumerate}
\item $\partial B_n (y)$ is contained in the union of stable manifolds for different Weyl chambers through points from $R$,
\item For each $n\geq 1$, $h(\partial B_n(y))$, is the boundary of a parallelepiped, moreover $diam  (h(\partial B_n(y)))\to 0$ as $n\to\infty$.
\end{enumerate}
\end{lem}
\begin{proof}
This is an immediate consequence of Lemma \ref{disconnect} and the traditional subdivision of boxes like in Heine-Borel theorem.
\end{proof}

\proof{\bf of Theorem \ref{TMain2}}
From Lemma \ref{boxeverypoint} it follows that $h$ uniquely extends to a continuous map from $B$ onto $\hat B$. Indeed for $y\in B$ take the nested sequence from Lemma \ref{boxeverypoint} and define $h(y)$ to be the limit point of $h(\partial B_n(y))$. Continuity follows since the preimage of the box bounded by $h(\partial B_n(y))$ is $B_n(y)$ that is a neighborhood of $y$ for every $n\geq 0$. 

From the definition of $O$ and $L\setminus F$ we get that $h$ extends uniquely to a continuous map $h:O\to L\setminus F$ that semi-conjugates. Moreover, from Lemma \ref{boxeverypoint} it also follows that for any $z\in L\setminus F$, $h^{-1}(z)$ is the nested intersection of boxes and for $\lambda$ a.e. $z\in L$ this nested intersection is a point by Lemma \ref{boxes}.

For the rest of the proof of Theorem  \ref{TMain2} we shall prove that the restriction of $h$ to a suitable $k$-dimensional skeleton is a diffeomorphism and that this restriction extends to a homeomorphism of $O$ onto $L\setminus F$.

We have the following topological lemma for the infratorus.
\begin{lem}\label{partition}
Given $\epsilon >0$, and a box $\hat B\subset L\setminus F$ as in Lemma \ref{boxes} there is a bounded subset $K\subset \Zmm$, $R> 0$ and a partition by rectangles $C_i$, $i=1,\dots r$ of the complement of some neighborhood of the singularities $L_{\epsilon}:=\bigcup_{1\leq i\leq r} C_i$, such that:
\begin{enumerate}
\item $diam C_i<\epsilon$,
\item $C_i\cap C_ j\subset\partial C_i\cap\partial C_j$ for $i\neq j$,
\item $C_i\subset \a_0(\bn)(\hat B)$ for some $\bn\in K$,
\item $$\partial C_i\subset \bigcup _{\substack{\bn\in K,\\ a\in\pm, \\1\leq j\leq m}}\a_0(\bn)((E^{s}_{\cC_j,R}+\partial^a_{\cC_j}(\hat B)))$$
\item $$L\setminus L_{\epsilon}\subset\bigcup_{z\in F}B_{\epsilon}(z),$$
\item $int L_{\epsilon}$ is homeomorphic to $L\setminus F$. 
\end{enumerate}

\end{lem}

\begin{proof}
Consider $L\setminus \bigcup_{z\in F}B_{\epsilon}(z)$ and the covering of this compact set by the iterates of $\hat B$, $\a_0(\bn)(\hat B)$, $\bn\in\mathbb{Z}^{m-1}$. Take a finite subcover, i.e. a finite subset $K\subset \mathbb{Z}^{m-1}$ so that $\a_0(\bn)(\hat B)$, $\bn\in K$ also covers. Now, $L_{\epsilon}=\bigcup_{\bn\in K}\a_0(\bn)(\hat B)$ admits a partition by rectangles $R_i$ as desired. 
\end{proof}

Let $\widehat {Sk}=\bigcup_i \partial C_i$ be the $m-1$-dimensional skeleton defined by the partition from Lemma \ref{partition}. 

\begin{lem}\label{extension}
$h^{-1}$ restricted to $\widehat {Sk}$ is a diffeomorphism onto a $k$-dimensional skeleton $Sk$. Moreover the diffeomorphism $h^{-1}:Sk\to\widehat{Sk}$ extends to a homeomorphism $g_{\epsilon}:L_{\epsilon}\to U_{\epsilon}$ from $L_\epsilon$ onto an open subset $U_{\epsilon}\subset O$,  that is a diffeomorphism if $m-1=2, 4, 5, 11, 60$, i.e. $m=3, 5, 6, 12, 61$, see \cite{jmilnor2}.  
\end{lem}

Observe that for other dimensions $m-1$ the possible existence of exotic spheres and hence of nonstandard smooth embeddings of $S^{m-1}$ into $\mathbb R^{m}$, \cite{jmilnor,kervairemilnor} may  preclude the possibility of extending $h^{-1}$ diffeomorphically to some cell of the partition. 

\begin{proof}
That $h^{-1}$ restricted to $\widehat {Sk}$ is a diffeomorphism is a consequence of Lemma \ref{smoothconjugacy}. Hence we have a well defined skeleton $$Sk=h^{-1}(\widehat {Sk})=\bigcup _ih^{-1}(\partial C_i).$$ Since $C_i\subset\a_0(\bn)(\hat B)$ for some $\bn\in K$ we have that $h^{-1}(\partial C_i)\subset\a(\bn)(B)$ for some $\bn\in K$. Hence $h^{-1}:\partial C_i\to \a(\bn)(B)$ is an embedding of the $m-1$-dimensional sphere into the a $m$ dimensional cube $\a(\bn)(B)$. Now, Sch\"onflies Theorem and Alexander trick gives that $h^{-1}$ extends to a homeomorphism. The differentiable statement follows from the smooth Sch\"onflies theorem, valid for $m-1\neq 3$, plus the nonexistence of exotic embeddings for the given dimensions.  
\end{proof}

Lemma \ref{extension} and that $int L_\epsilon$ is diffeomorphic to $L\setminus F$ finishes the proof of Theorem \ref{TMain2}.


\begin{thebibliography}{20}

\bibitem{jalexander} J. Alexander, {\em On the subdivision of space by a polyhedron.} Proc. Nat. Acad. Sci. USA {\bf 10} (1924) pp. 6--8.

\bibitem{dberend} D. Berend, {\em Multi-invariant sets on tori}, Trans. Amer. Math. Soc, {\bf 280} (1983), 509--532.

\bibitem{mbrown} M. Brown, {\em A proof of the generalized Schoenflies theorem.} Bull. Amer. Math. Soc.,  {\bf 66}, (1960) pp. 74--76. 

\bibitem{dlL} R. de la Llave, {\em Smooth conjugacy and SRB measures for uniformly and non-uniformly hyperbolic systems} Comm. Math. Phys., {\bf 150} (1992), 289--320.

\bibitem{FKS}  D. Fisher, B. Kalinin and R. Spatzier, {\em Global rigidity of higher rank Anosov actions on tori and nilmanifolds. With an appendix by James F. Davis.} J. Amer. Math. Soc. {\bf 26} (2013), no. 1, 167--198.

\bibitem{HS} G. H\"ohn, N.-P. Skoruppa, {\em Un r\'esultat de Schinzel}, J. Th\'eor. Nombres Bordeaux {\bf 5} 
(1993), 185. 


\bibitem{Hopf} H. Hopf, {Systeme symmetrischer Bilinearformen und euklidische Modelle der projektiven R\"aume,} Vierteljschr. Naturforsch. Gesellschaft Zurich {\bf 85} (1940) 165--177.

\bibitem{Hhu}  H. Hu {\em Some ergodic properties of commuting diffeomorphisms}, Ergodic Theory Dynam. Systems {\bf 13} (1993), no. 1, 73--100. 

\bibitem{KK} B. Kalinin and  A. Katok,  {\em Measure rigidity beyond uniform hyperbolicity: Invariant Measures for Cartan actions on Tori},  Journal of Modern Dynamics, {\bf 1} N1 (2007),
123--146.


\bibitem{KKRH}  B. Kalinin, A. Katok and F. Rodriguez Hertz, {\em Nonuniform Measure Rigidity}, Annals of mathematics, {\bf 174}, (2011), 361--400.

\bibitem{K} A.Katok,  {\em Lyapunov exponents, entropy and periodic points of diffeomorphisms}, Publ. Math. IHES, {\bf51}, (1980), 137-173.


\bibitem{AKSKKS}  A. Katok  S. Katok and K. Schmidt,  {\em  Rigidity
of measurable structure for $ Z^d$ actions by automorphisms of a
torus},  Comm. Math. Helvetici, {\bf 77},  (2002),  718-745.

\bibitem{KL} A.Katok and J. Lewis, {\em Global rigidity results for lattice actions on tori and new examples of volume-preserving actions},  Israel J.Math., {\bf93}, (1996), 253-280.

\bibitem{KNbook} A. Katok and V. Nitica, {\em Rigidity in higher rank abelian actions; vol 1. Introduction and cocycle problem}, Cambridge University Press, 2011. 

\bibitem{KRH1} A. Katok and F. Rodriguez Hertz, {\em Uniqueness of large invariant measures for $\Zk$ actions with Cartan homotopy data}, Journal of Modern Dynamics, {\bf 1} (2007), 287--300.


\bibitem{KRH2} A. Katok and  F. Rodriguez Hertz, {\em Measure and cocycle rigidity for certain non-uniformly hyperbolic actions of higher rank abelian groups}, Journal of Modern Dynamics, {\bf 4}, N3, (2010), 487--515. 

\bibitem{KRHZimmer} A. Katok and  F. Rodriguez Hertz, {\em Rigidity of real-analytic actions of $SL(n,\mathbb Z)$ on $\mathbb T^n$ : A case of realization of Zimmer program},  Discrete and Continuous Dynamical Systems, {\bf 27} (2010), 609--615

\bibitem {kervairemilnor} M. Kervaire, J. Milnor, {\em Groups of homotopy spheres, I}, Ann. of Math.
{\bf 77} (1963), pp.504--537.



\bibitem{jmilnor} J. Milnor, {\em On manifolds homeomorphic to the 7-sphere}, Ann. of Math., {\bf 64},  (1956), 399--405.

\bibitem{jmilnor2}  J. Milnor, {\em Differential topology forty-six years later.} Notices Amer. Math. Soc. {\bf 58} (2011), no. 6, 804--809

\bibitem{mmorse} M. Morse, {\em A reduction of the Schoenflies extension problem.} Bull. Amer. Math. Soc. {\bf 66} (1960) 113--115. 

\bibitem{bmazur} B. Mazur, {\em On embeddings of spheres.}, Bull. Amer. Math. Soc. {\bf 65} (1959) 59--65. 

\bibitem{Parry}  W. Parry,  {\em Synchronization of canonical measures for hyperbolic attractors}, Comm. Math. Phys. {\bf 106} (1986), no. 2, 267--275.

\bibitem{P} Ya. Pesin, {\em Characteristic Ljapunov exponents, and smooth ergodic theory.} (Russian) Uspehi Mat. Nauk {\bf 32} (1977), no. 4 (196), 55--112. 

\bibitem{hhtu} F. Rodriguez Hertz, J. Rodriguez Hertz, A. Tahzibi, R. Ures, {\em New criteria for ergodicity and nonuniform hyperbolicity.} Duke Math. J. {\bf 160} (2011), no. 3, 599--629.

\bibitem{rhglobal}  F. Rodriguez Hertz, {\em Global rigidity of certain abelian actions by toral automorphisms.} J. Mod. Dyn. {\bf 1} (2007), no. 3, 425--442.

\bibitem{RH-W}  F. Rodriguez Hertz, Z. Wang {\em Global rigidity of higher rank abelian Anosov algebraic actions} (preprint arXiv:1304.1234)

\bibitem{Schnizel}A. Schinzel,  {\em On the product of the conjugates outside the unit circle of an algebraic 
number}, Acta Arith. {\bf24}  (1973), 385Ð399; Addendum:{\bf 26} (1974/75), 329Ð331. 


\end{thebibliography}
\end{document}